\documentclass[a4paper,10pt%,draft
]{amsart}

\usepackage{amsmath,amssymb}

\usepackage[alphabetic, abbrev]{amsrefs}

\usepackage{enumerate}
\usepackage{hyperref}
\usepackage[all]{xy}
\newdir{ >}{{}*!/-5pt/@{>}} %% cf xyguide exercise 14
\SelectTips{cm}{} %% cf xyguide page 9

\numberwithin{equation}{section}
\newtheorem{theorem}[subsection]{Theorem}
\newtheorem{corollary}[subsection]{Corollary}
\newtheorem{lemma}[subsection]{Lemma}
\newtheorem{proposition}[subsection]{Proposition}
\theoremstyle{definition}
\newtheorem{definition}[subsection]{Definition}
\newtheorem{remark}[subsection]{Remark}
\newtheorem{example}[subsection]{Example}

\newtheorem{construction}[subsection]{Construction}

\newcommand{\bC}{\mathbb{C}}

\newcommand{\bN}{\mathbb{N}}
\newcommand{\bS}{\mathbb{S}}
\newcommand{\bZ}{\mathbb{Z}}
\newcommand{\bF}{\mathbb{F}}
\newcommand{\bQ}{\mathbb{Q}}

\newcommand{\bm}{\mathbf{m}}
\newcommand{\bn}{\mathbf{n}}
\newcommand{\bzero}{\mathbf{0}}

\newcommand{\cB}{\mathcal{B}}
\newcommand{\cC}{\mathcal{C}}
\newcommand{\cE}{\mathcal{E}}
\newcommand{\cI}{\mathcal{I}}
\newcommand{\cJ}{\mathcal{J}}
\newcommand{\cS}{\mathcal{S}}
\newcommand{\SJ}{\cS^{\cJ}}
\newcommand{\CSJ}{\cC\cS^{\cJ}}

\newcommand{\cof}{\rightarrowtail}

\newcommand{\longto}{\longrightarrow}

\DeclareMathOperator{\hocolim}{hocolim}

\DeclareMathOperator{\THH}{THH}
\DeclareMathOperator{\TC}{TC}
\DeclareMathOperator{\Mod}{Mod}

\DeclareMathOperator{\Tor}{Tor}

\renewcommand{\:}{\colon}

\newcommand{\ot}{\leftarrow}
\newcommand{\sm}{\wedge}
\newcommand{\iso}{\cong}
\newcommand{\tensor}{\otimes}
\newcommand{\bld}[1]{{\mathbf{#1}}}
\newcommand{\gp}{\mathrm{gp}}
\newcommand{\rep}{\mathrm{rep}}
\newcommand{\cy}{\mathrm{cy}}

\newcommand{\Spsym}{{\mathrm{Sp}^{\Sigma}}}

\DeclareMathOperator{\capitalGL}{GL}

\newcommand{\GLoneJof}[1]{\capitalGL^{\cJ}_1\!\!{#1}}
\newcommand{\concat}{\sqcup}
\DeclareMathOperator{\colim}{colim}
\newcommand{\modmod}[2]{#1/(#2_{>0})}
\newcommand{\modcofelement}[2]{#1^{\mathrm{cof}}\wedge_{\bS^{\cJ}[\bC\langle #2\rangle^{\mathrm{cof}}]}\bS}
\newcommand{\Fzero}{F}
\newcommand{\C}[1]{\bC\langle{#1}\rangle}
\newcommand{\postnikovsec}[3]{#1[{#2},{#3}\rangle}
\newcommand{\Cone}{C}

\newcommand{\arxivlink}[1]{\href{http://arxiv.org/abs/#1}{\texttt{arXiv:#1}}}
\begin{document}

\title[Localization sequences for log
\texorpdfstring{$\THH$}{THH}]{Localization sequences for\\ logarithmic
  topological Hochschild homology}

\author{John Rognes}
\address{John Rognes, Department of Mathematics, University of Oslo,  Box 1053 Blindern, 0316 Oslo, Norway}
\email{rognes@math.uio.no}

\author{Steffen Sagave} \address{Steffen Sagave, Department of Mathematics and Informatics, Bergische Universit{\"a}t Wuppertal,  Gau{\ss}\-str. 20, 42119 Wuppertal,
% Mathematical Institute, University of Bonn, Endenicher Allee 60, 53115 Bonn,
Germany} \email{sagave@math.uni-wuppertal.de
%sagave@math.uni-bonn.de
}

\author{Christian Schlichtkrull} \address{Christian Schlichtkrull,
    Department of Mathematics, University of Bergen, P.O. Box 7800, N-5020 Bergen, Norway} \email{christian.schlichtkrull@math.uib.no}

\date{\today}

\begin{abstract}
  We study the logarithmic topological Hochschild homology of ring spectra
  with logarithmic structures and establish localization sequences for
  this theory. Our results apply, for example, to connective covers of
  periodic ring spectra like real and complex topological $K$-theory.
%\subclass{14F10, 19D55, 55P43}
%\keywords{logarithmic ring spectrum \and replete bar construction}
\end{abstract}
\maketitle

\section{Introduction}
Algebraic $K$-theory provides a powerful invariant encoding deep
arithmetic properties.  For computations of algebraic $K$-theory of
rings it is often useful to invoke trace maps to topological
Hochschild homology ($\THH$) and to topological cyclic homology
($\TC$), since this makes tools from equivariant stable homotopy
theory applicable. This method is effective for rings satisfying
suitable finiteness conditions, as a consequence of Quillen's work for
finite fields~\cite{Quillen_finite-fields} and McCarthy's
theorem~\cite{McCarthy_Relative-K-TC}.  Important examples of this
approach are found in papers of
B{\"o}kstedt--Madsen~\cite{Boekstedt-M_TC-Z,Boekstedt-M_unramified}
and of Hesselholt--Madsen~\cite{Hesselholt-M_K-theoy-finite-alg-Witt}.

There are, however, examples of rings $A$ where the cyclotomic trace
map to $\TC(A)$ fails to provide a good approximation to the algebraic
$K$-theory $K(A)$. One explanation for this is that $\THH$ and $\TC$
do not admit the same localization sequences as algebraic
$K$-theory. We illustrate this with an example from the work of
Hesselholt--Madsen~\cite{Hesselholt-M_local_fields}: If $p$ is a prime
and $F$ is a finite field extension of $\mathbb{Q}_p$ with valuation
ring $A$ and residue field~$k$, then there is a localization homotopy
cofiber sequence of $K$-theory spectra
\begin{equation}\label{eq:Kth-loc-k-A-K} K(A) \to K(F) \to \Sigma K(k) \end{equation} 
established by Quillen~\cite{Quillen_higher}. Replacing $K$-theory by
$\THH$ or $\TC$, the corresponding diagrams do not form homotopy
cofiber sequences, and the trace maps from $K(F)$ to $\THH(F)$ and
$\TC(F)$ detect little information about $K(F)$, compared to the trace
maps from $K(A)$ and $K(k)$. In~\cite{Hesselholt-M_local_fields},
Hesselholt and Madsen overcome this by constructing relative forms
$\THH(A|F)$ and $\TC(A|F)$ of $\THH$ and $\TC$ that fit into a
localization homotopy cofiber sequence
\begin{equation}\label{eq:THH-loc-k-A-K} \THH(A) \to \THH(A|F) \to \Sigma \THH(k) \end{equation} and
a corresponding sequence for $\TC$, and they use $\TC(A|F)$ to
determine $K(F)$. While the definition of $\THH(A|F)$ given
in~\cite{Hesselholt-M_local_fields} uses linear Waldhausen categories,
the homotopy groups of $\THH(A|F)$ and $\TC(A|F)$ exhibit a close
connection to a logarithmic de\,Rham complex and a logarithmic
de\,Rham--Witt complex associated with the direct image logarithmic
structure on $A$ inherited from $F$.  This indicates a relation
between homotopy theory and logarithmic geometry in the sense
of~\cite{Kato_logarithmic-structures}.  A systematic investigation of
the interplay of these two subjects was taken up in the first author's
work on \emph{topological logarithmic
  structures}~\cite{Rognes_TLS}. The aim of the present paper is to
continue and extend this investigation with a focus on $\THH$ of
structured ring spectra.

\subsection{Algebraic \texorpdfstring{$K$}{K}-theory and
  \texorpdfstring{$\THH$}{THH} of structured ring spectra}
Trace maps to $\THH$ and $\TC$ also provide a good tool for computing
the algebraic $K$-theory of structured ring spectra, by Dundas'
theorem~\cite{Dundas_relative-K-TC}. They are, however, only directly
useful for \emph{connective} ring spectra satisfying suitable finiteness
conditions on~$\pi_0$.  Again we illustrate this shortcoming by an
example of a localization sequence. Analogously
to~\eqref{eq:Kth-loc-k-A-K}, there is a homotopy cofiber sequence
\begin{equation}\label{eq:Kth-loc-Z-ku-KU}  K(ku) \to K(KU)\to \Sigma K(\bZ) \end{equation} 
relating the algebraic $K$-theory spectra of the periodic complex
topological $K$-theory spectrum $KU$, its connective cover $ku$, and
the integers. The existence of this homotopy cofiber sequence was
conjectured by Ausoni and Rognes~\cite{Ausoni-R_Kku} and established
by Blumberg and Mandell~\cite{Blumberg-M_loc-sequenceKtheory}.
Replacing $K$-theory by $\THH$ in~\eqref{eq:Kth-loc-Z-ku-KU}, the
corresponding sequence of spectra fails to be a homotopy cofiber
sequence. To obtain a $\THH$ localization sequence analogous
to~\eqref{eq:THH-loc-k-A-K}, Blumberg and
Mandell~\cite{Blumberg-M_loc-sequenceTHH} have constructed a relative
$\THH$ term $\THH(ku | KU)$ that fits in a homotopy cofiber sequence
with $\THH(ku)$ and $\THH(\bZ)$. Their approach uses a version of
$\THH$ for simplicially enriched Waldhausen categories, by analogy
with~\cite{Hesselholt-M_local_fields}.

In the present paper we offer a different approach to such relative
$\THH$ terms by defining them as the \emph{logarithmic $\THH$} of
certain \emph{logarithmic ring spectra}, to be introduced
next. Compared to the relative $\THH$ terms defined by Blumberg and
Mandell, our approach is more directly connected to logarithmic
de\,Rham and de\,Rham--Witt complexes, more amenable to homology
computations, and applies to new examples including the real
topological $K$-theory spectrum. This answers questions that remained
open in~\cite{Rognes_TLS}, and builds on more recent foundational work
by the last two authors~\cite{Sagave-S_diagram,Sagave_spectra-of-units}.

\subsection{Logarithmic ring spectra}
A \emph{pre-log ring} $(A,M)$ is a commutative ring $A$ together with
a commutative monoid $M$ and a monoid homomorphism $\alpha\colon M \to
(A,\cdot)$ to the multiplicative monoid of $A$. It is a log ring if
the base change $\alpha^{-1}(A^{\times}) \to A^{\times}$ of $\alpha$
along the inclusion of the units $A^{\times} \to (A,\cdot)$ is an
isomorphism. If $A$ is any commutative ring, we can use $A^{\times} \to
(A,\cdot)$ to form the \emph{trivial} log ring $(A,A^{\times})$. A log
ring $(A,M)$ determines a localization $A[M^{-1}]$, and the map of
trivial log rings associated with $A \to A[M^{-1}]$ factors through
$(A,M)$ in a non-trivial way.

In order to form a homotopical generalization of pre-log rings, one
can consider commutative symmetric ring spectra $A$ together with maps
of commutative $\cI$-space monoids $M \to \Omega^{\cI}(A)$ in the sense of \cite[Section 3]{Sagave-S_diagram}.  Here,
commutative $\cI$-space monoids give one possible model for the more
commonly studied $E_{\infty}$ spaces, and may be viewed as a
homotopical counterpart of commutative monoids. The commutative
$\cI$-space monoid $\Omega^{\cI}(A)$ encodes the underlying
multiplicative $E_{\infty}$ space of~$A$. This $\cI$-space version of
\emph{pre-log ring spectra} was considered in~\cite{Rognes_TLS}. It
has the disadvantage that it appears to be difficult to extend $A$ to
a pre-log ring spectrum $(A,M)$ in a sufficiently interesting way if
$A$ is not an Eilenberg--Mac\,Lane spectrum. One reason for this is that
commutative $\cI$-space monoids and $E_\infty$ spaces are inherently
connective, hence cannot be group completed in such a way that
positive dimensional homotopy classes are inverted. To remedy this, we
proceed as in~\cite[\S 4.30]{Sagave-S_diagram}
and~\cite{Sagave_log-on-k-theory} and replace the commutative
$\cI$-space monoids in the previous definition by the commutative
$\cJ$-space monoids developed by the last two authors.

Let $\cJ$ be the category $\Sigma^{-1}\Sigma$ given by Quillen's
localization construction~\cite{Grayson_higher} on the category
$\Sigma$ of finite sets and bijections. A \emph{commutative
  $\cJ$-space monoid} is a lax symmetric monoidal functor from $\cJ$
to the category of unbased spaces~$\cS$.  Equivalently, it is a
commutative monoid object with respect to a convolution product on the
functor category $\cS^{\cJ}$. The resulting category $\cC\cS^{\cJ}$ of
commutative $\cJ$-space monoids admits a model structure making it
Quillen equivalent to the category of $E_{\infty}$ spaces over the
underlying additive $E_{\infty}$ space $QS^0$ of the sphere
spectrum. We therefore think of commutative $\cJ$-space monoids as a
model for ($QS^0$-)\emph{graded $E_{\infty}$ spaces}.

For a commutative symmetric ring spectrum $A$, one can form a
commutative $\cJ$-space monoid $\Omega^{\cJ}(A)$ that is a graded
version of the multiplicative $E_{\infty}$ space of~$A$. There is a
sub commutative $\cJ$-space monoid $\GLoneJof{(A)}$ of $\Omega^{\cJ}(A)$
that corresponds to the inclusion of multiplicative units
$\pi_*(A)^{\times} \subset\pi_*(A)$. In contrast, the usual
$E_{\infty}$ space of units of $A$ only corresponds to the inclusion
$\pi_0(A)^{\times} \subset\pi_0(A)$ of units in degree
$0$. 

A \emph{pre-log ring spectrum} is then a commutative symmetric ring
spectrum $A$ together with a commutative $\cJ$-space monoid $M$ and a
map $\alpha\colon M \to \Omega^{\cJ}(A)$ of commutative $\cJ$-space
monoids. It is a \emph{log ring spectrum} if the base change
$\alpha^{-1} \GLoneJof(A) \to \GLoneJof(A)$ of its structure map
$\alpha$ along $ \GLoneJof{(A)}\to\Omega^{\cJ}(A)$ is a weak
equivalence of $\cJ$-spaces. The easiest example of a log ring
spectrum is the trivial log ring spectrum $(A, \GLoneJof(A))$
associated to any commutative symmetric ring spectrum $A$. A more
elaborate example is given by the following construction, which plays
an important role in this paper: If $j\colon e \to E$ is the
connective cover map of a periodic commutative symmetric ring spectrum
$E$, then we let $j_*\!\GLoneJof(E)$ be the pullback of the following
diagram of commutative $\cJ$-space monoids:
\[\GLoneJof(E) \to \Omega^{\cJ}(E)
\ot \Omega^{\cJ}(e).\] Together with the canonical map
$j_*\!\GLoneJof(E) \to \Omega^{\cJ}(e)$ from the pullback, this defines
a log ring spectrum $(e, j_*\!\GLoneJof(E))$. In analogy with a
similar construction in algebraic geometry, we call this the
\emph{direct image} log ring spectrum associated with the trivial log
ring spectrum $(E, \GLoneJof(E))$. We note that if we were using
the $\cI$-space version of pre-log ring spectra the same construction
would only provide a trivial log structure, because the map of
$\cI$-space units associated with $e \to E$ is an equivalence.

It follows from the definition that the map of trivial log ring
spectra associated with $e \to E$ factors through $(e,
j_*\!\GLoneJof(E))$.  One indication for why $(e, j_*\!\GLoneJof(E))$ is
interesting is the following result proved in~\cite[Theorem
4.4]{Sagave_log-on-k-theory}: If $E$ is periodic, it can be recovered as the trivial locus of $(e, j_*\!\GLoneJof(E))$ (see~\cite[Definition 7.15]{Rognes_TLS}) by forming the homotopy pushout of
\[ \bS^{\cJ}[(j_*\!\GLoneJof(E))^{\gp}] \ot \bS^{\cJ}[j_*\!\GLoneJof(E)] \to
e. \] Here $\bS^{\cJ}$ denotes the graded spherical monoid ring
functor that is left adjoint to $\Omega^{\cJ}$, the right hand
map is the adjoint of the structure map of $(e, j_*\!\GLoneJof(E))$, and
the left hand map is induced by the group completion of
$j_*\!\GLoneJof(E)$ as defined in~\cite{Sagave_spectra-of-units}.

\subsection{Logarithmic \texorpdfstring{$\THH$}{THH}}
If $A$ is a commutative symmetric ring spectrum, then its topological
Hochschild homology $\THH(A)$ can be defined as the realization of the
\emph{cyclic bar construction} $[q] \mapsto A^{\sm (q+1)}$. For a
commutative $\cJ$-space monoid $M$, one can define $B^{\cy}(M)$ as the realization of the analogous cyclic bar construction
$[q] \mapsto M^{\boxtimes (q+1)}$, where $\boxtimes$ is the convolution
product of $\cJ$-spaces.  If $(A,M)$ is a pre-log ring spectrum, then
the adjoint structure map $\bS^{\cJ}[M] \to A$ induces a map
$\bS^{\cJ}[B^{\cy}(M)] \to \THH(A)$ of commutative symmetric ring spectra.

The last ingredient in the definition of logarithmic $\THH$ is the
\emph{replete bar construction} $B^{\rep}(M)$. This is a subtle
variant of the cyclic bar construction. One motivation for using the
replete bar construction is that its algebraic counterpart can be used
to define the logarithmic Hochschild homology of log rings, which in
the log smooth case agrees with the logarithmic de\,Rham
complex~\cite[\S 3, \S 5]{Rognes_TLS}. Using the group completion $M
\to M^{\gp}$ for commutative $\cJ$-space monoids constructed
in~\cite{Sagave_spectra-of-units}, the commutative $\cJ$-space monoid
$B^{\rep}(M)$ is defined as the homotopy pullback of the diagram $M
\to M^{\gp} \ot B^{\cy}(M^{\gp})$. It comes with a canonical \emph{repletion
map} $\rho\colon B^{\cy}(M)\to B^{\rep}(M)$.

The logarithmic $\THH$ of a pre-log ring spectrum 
$(A,M)$ is then defined as the homotopy pushout of the following diagram of
commutative symmetric ring spectra:
\[ \bS^{\cJ}[B^{\rep}(M)] \ot \bS^{\cJ}[B^{\cy}(M)] \to \THH(A).\] It
is not difficult to see that the canonical map $\THH(A) \to \THH(A,M)$
is a stable equivalence if $M$ is grouplike. This applies in
particular for the trivial log ring spectrum $(A, \GLoneJof(A))$. A
useful but more involved property of log $\THH$ is its invariance
under logification. This means that the \emph{logification map} $(A,M)
\to (A^a,M^a)$, which naturally associates a log ring spectrum $(A^a,M^a)$
to each pre-log ring spectrum $(A,M)$, induces a stable equivalence
$\THH(A,M) \to \THH(A^a,M^a)$.

Our main theorem states that under a certain condition on $M$, the
logarithmic $\THH$ of a pre-log ring spectrum $(A,M)$ participates in
a localization homotopy cofiber sequence, where the two other terms
are given by ordinary topological Hochschild homology.  To formulate
the condition, we use that every $\cJ$-space inherits a $\bZ$-grading from the isomorphism $\pi_0(B\cJ)\iso \bZ$. A commutative $\cJ$-space monoid 
$M$ is said to be \emph{repetitive} if it is $\cJ$-equivalent to the
non-negative part of its group completion~$M^{\gp}$, and if in addition the positive part of $M$ is nonempty.
\begin{theorem}\label{thm:THH-cofiber-sequence-with-repetitive}
  Let $(A,M)$ be a pre-log ring spectrum with $M$ repetitive.
  Then there is a natural homotopy cofiber sequence
  \[\THH(A) \xrightarrow{\rho}
  \THH(A,M) \xrightarrow{\partial} \Sigma
  \THH(\modmod{A}{M})\] of $\THH(A)$-module spectra
  with circle action.
\end{theorem}
In the theorem, the commutative symmetric ring spectrum
$\modmod{A}{M}$ is the homotopy pushout of the diagram $A \ot
\bS^{\cJ}[M] \to \bS^{\cJ}[M_{\{0\}}]$ induced by the adjoint
structure map of $(A,M)$ and the map $\bS^{\cJ}[M] \to
\bS^{\cJ}[M_{\{0\}}]$ that collapses the positive degree parts of $M$.

In examples of interest, we can describe $\modmod{A}{M}$ more
explicitly. Let $E$ be a commutative symmetric ring spectrum such that
$0\neq 1$ in $\pi_0(E)$. We say that $E$ is \emph{$d$-periodic} if
$\pi_*(E)$ has a unit of positive degree and $d$ is the minimal degree
of such a unit. If $E$ is $d$-periodic and $j\colon e \to E$ is the
connective cover, then the commutative $\cJ$-space monoid
$j_*\GLoneJof(E)$ participating in the log ring spectrum
$(e,j_*\GLoneJof(E))$ is repetitive, and $\modmod{e}{j_*\GLoneJof(E)}$
is stably equivalent to the $(d-1)$-th Postnikov section
$\postnikovsec{e}{0}{d}$ of~$e$.  In this situation
Theorem~\ref{thm:THH-cofiber-sequence-with-repetitive} leads to the
following statement:

\begin{theorem}\label{thm:THH-cofiber-sequence-with-direct-image}
  Let $E$ be a $d$-periodic commutative symmetric ring spectrum with
  connective cover $j\colon e \to E$. There is a natural homotopy
  cofiber sequence
  \[\THH(e) \xrightarrow{\rho}
  \THH(e,j_*\!\GLoneJof(E)) \xrightarrow{\partial} \Sigma
  \THH(\postnikovsec{e}{0}{d})\] of $\THH(e)$-module spectra
  with circle action.
\end{theorem}
The theorem applies, for example, to the $8$-periodic real $K$-theory
spectrum $KO$ and its connective cover $j\colon ko \to KO$, where we
obtain a homotopy cofiber sequence
\[ \THH(ko) \xrightarrow{\rho} \THH(ko,j_*\!\GLoneJof(KO))
\xrightarrow{\partial} \Sigma \THH(\postnikovsec{ko}{0}{8}). 
\]
The analogy with the homotopy cofiber
sequence~\eqref{eq:THH-loc-k-A-K} indicates that one may view the
Postnikov section $\postnikovsec{ko}{0}{8}$ as a nilpotent extension
of the \emph{residue ring (spectrum)} of $ko$. Consequently, one may wonder about the $K$-theoretic significance of $\postnikovsec{ko}{0}{8}$, and we expect that $\postnikovsec{ko}{0}{8} \to H\bZ$ will induce an equivalence in $G$-theory (compare~\cite{Barwick-L_regularity}). 

In the case of the $p$-local complex topological $K$-theory spectra
$ku_{(p)}\to KU_{(p)}$, Theorem~\ref{thm:THH-cofiber-sequence-with-direct-image}
provides a homotopy cofiber sequence
\[ \THH(ku_{(p)}) \xrightarrow{\rho} \THH(ku_{(p)},j_*\!\GLoneJof(KU_{(p)}))
\xrightarrow{\partial} \Sigma \THH(\bZ_{(p)}), 
\]
and similarly for the Adams summand $\ell$ of $ku_{(p)}$ and the map
$j\colon \ell \to L$ to its periodic counterpart. In a
sequel~\cite{RSS_LogTHH-II} to this paper, we will determine the
$V(1)$-homotopy of $\THH(\ell, j_*\!\GLoneJof(L))$, show that the
inclusion of the Adams summand $\ell \to ku_{(p)}$ induces a stable
equivalence
\[ku_{(p)}\sm_{\ell}\THH(\ell, j_*\!\GLoneJof(L)) \to \THH(ku_{(p)},
j_*\!\GLoneJof(KU_{(p)})), \] and use this to calculate the
$V(1)$-homotopy of $\THH(ku_{(p)}, j_*\!\GLoneJof(KU_{(p)}))$. In this way we complete the conjectural program outlined by Ausoni and Hesselholt for 
simplifying Ausoni's computation in~\cite{Ausoni_THH-ku} of the
$V(1)$-homotopy of $\THH(ku_{(p)})$.

The construction of the localization sequence in 
 Theorem~\ref{thm:THH-cofiber-sequence-with-repetitive} is based on a general principle that also applies to pre-log ring spectra that arise from pre-log
rings $(B,N)$ in the algebraic sense. Let
$B$ be a commutative ring and let $\beta \colon N \to (B,\cdot)$ be such
that $N$ is a free commutative monoid on one generator and $\beta$ maps that
generator to an $x \in B$ that does not divide zero. Then there is a homotopy
cofiber sequence
\[\THH(B) \to \THH(B,N) \to \Sigma \THH(B/(x)).  \]
In Section~\ref{sec:discrete-pre-log-rings} we calculate the mod $p$
homotopy of $\THH(B,N)$ in the case $B = \bZ_{(p)}$ and $x = p$ and
show that it agrees with that of Hesselholt and Madsen's construction
$\THH(\bZ_{(p)}|\bQ)$.

\subsection{Notation and conventions}
We assume some familiarity with model categories, and mostly use
Hirschhorn's book~\cite{Hirschhorn_model} as a reference. In
particular, we frequently use the notions of homotopy cartesian and
cocartesian squares in proper model categories; see e.g.~\cite[\S
13]{Hirschhorn_model}. When working with symmetric spectra, we shall
use both the simplicial version introduced in \cite{HSS} and the
topological version discussed in \cite{MMSS}. Given a symmetric ring
spectrum $A$, we shall use the expression: \emph{a homotopy cofiber
  sequence
\[X\xrightarrow{f} Y \xrightarrow{\partial} Z
\] 
of $A$-modules}, to mean a map $f\colon X \to Y$ of $A$-module spectra
together with an $A$-module spectrum $Z$ and an implicit chain of
stable equivalences of $A$-module spectra between the mapping cone
$C(f)$ and $Z$, all of this understood internally to the category of
symmetric spectra. So, by abuse of notation, $\partial$ denotes the
canonical map $Y \to C(f)$ followed by the chain of stable
equivalences. To avoid keeping track of semistability and fibrancy of
symmetric spectra, we use the notation $\pi_*(A)$ for the stable
homotopy groups of a (positive) fibrant replacement of a symmetric
spectrum~$A$.

\subsection{Organization}
We begin in Section~\ref{sec:Spsym-SJ} with a brief review of
$\cJ$-spaces and their relation to symmetric spectra. In
Section~\ref{sec:cy-and-rep} we recall the definition of $\THH$ in the
setting of symmetric ring spectra and introduce the cyclic and replete
bar constructions of commutative $\cJ$-space monoids.
Section~\ref{sec:log-thh} contains the definition of (pre-)log ring
spectra and their log $\THH$, and we prove the invariance of log
$\THH$ under logification.  In
Section~\ref{sec:discrete-pre-log-rings} we study the logarithmic
$\THH$ of pre-log ring spectra arising from the algebraic version of
pre-log rings, and we set up the relevant localization sequences in
this case.  In Section~\ref{sec:hty-cof-log-thh} we turn to repetitive
pre-log ring spectra and construct the localization sequences in
Theorems~\ref{thm:THH-cofiber-sequence-with-repetitive}
and~\ref{thm:THH-cofiber-sequence-with-direct-image} from the
introduction.  The final Section~\ref{sec:proposition-proof} contains
the proof of the main result about homotopy cofiber sequences needed
to prove Theorem~\ref{thm:THH-cofiber-sequence-with-repetitive}. An
Appendix collects homotopy invariance properties of the functor
$\bS^{\cJ}$ from $\cJ$-spaces to symmetric spectra.

\subsection{Acknowledgments} The authors would like to thank the
  referee for useful comments.

\section{Symmetric spectra and \texorpdfstring{$\cJ$}{J}-spaces}\label{sec:Spsym-SJ}
In this section we review the definition of $\cJ$-spaces and
commutative $\cJ$-space monoids from~\cite{Sagave-S_diagram} and
explain their relation to symmetric spectra.

The category of symmetric spectra $\Spsym$ introduced in~\cite{HSS} is
a stable model category whose homotopy category is the stable homotopy
category. It has a symmetric monoidal smash product denoted by $\sm$
whose monoidal unit is the sphere spectrum~$\bS$.  The commutative
monoids with respect to the smash product are known as
\emph{commutative symmetric ring spectra}. They may be viewed as
strictly commutative models for $E_{\infty}$ ring spectra. We will use
that the category of commutative symmetric ring spectra $\cC\Spsym$
inherits a proper simplicial \emph{positive stable model structure}
from $\Spsym$~\cite{MMSS}. The book project~\cite{Schwede_SymSp}
provides extensive background about symmetric spectra.

\subsection{\texorpdfstring{$\cJ$}{J}-spaces and commutative
  \texorpdfstring{$\cJ$}{J}-space monoids} 
We recall from~\cite{Sagave-S_diagram} how one can use commutative
monoid objects in the category of space valued functors on an
appropriate indexing category as a model for a graded version of
$E_{\infty}$ spaces.

\begin{definition}[{\cite[Definition 4.2]{Sagave-S_diagram}}]
  Let $\cJ$ be the category whose objects are pairs
  $(\bld{m_1},\bld{m_2})$ of finite sets $\bld{m_i}=\{1,\dots,m_i\}$
  with each $m_i\geq 0$. A morphism \[(\alpha_1,\alpha_2,\rho)\colon
  (\bld{m_1},\bld{m_2}) \to (\bld{n_1},\bld{n_2})\] in $\cJ$ consists of two
  injective functions $\alpha_i\colon \bld{m_i} \to \bld{n_i}$ and a
  bijection $\rho \colon \bld{n_1}\setminus \alpha_1 \to
  \bld{n_2}\setminus \alpha_2$ identifying the complement of the image
  of $\alpha_1$ with the complement of the image of
  $\alpha_2$. Consequently, the set of morphisms from
  $(\bld{m_1},\bld{m_2})$ to $(\bld{n_1},\bld{n_2})$ is empty unless
  $m_2-m_1 = n_2 -n_1$.
\end{definition}
It is proven in~\cite[Proposition 4.4]{Sagave-S_diagram} that $\cJ$ is
isomorphic to Quillen's localization construction $\Sigma^{-1}\Sigma$
on the permutative category $\Sigma$ of finite sets and
bijections. Combining this with the Barratt--Priddy--Quillen theorem
shows that the classifying space $B\cJ$ of $\cJ$ is weakly equivalent
to $QS^0 = \Omega^{\infty}\Sigma^{\infty}S^0$ as an infinite loop
space.

\begin{definition}
  A \emph{$\cJ$-space} is a functor $X \colon \cJ \to \cS$ from $\cJ$
  to the category of unbased simplicial sets $\cS$. The functor
  category of $\cJ$-spaces is denoted by~$\cS^{\cJ}$.
\end{definition}
The ordered concatenation $-\concat -$ of sets in both entries
makes $\cJ$ a symmetric monoidal category. Its monoidal unit is
$(\bld{0},\bld{0})$.  Defining $X \boxtimes Y$ to be the left Kan
extension of the object-wise product along $ - \concat -\colon \cJ
\times \cJ \to \cJ$ makes $\cS^{\cJ}$ a symmetric monoidal category with
product $\boxtimes$, unit $U^{\cJ} = \cJ((\bld{0},\bld{0}),-)$, and symmetry
isomorphism $\tau \colon X\boxtimes Y \to Y \boxtimes X$. 

\begin{definition} A \emph{commutative $\cJ$-space monoid} is a
  commutative monoid in $(\cS^{\cJ},\boxtimes, U^{\cJ},\tau)$, and $\CSJ$
  denotes the category of commutative $\cJ$-space monoids.
\end{definition}
Although we are mostly concerned with commutative $\cJ$-space monoids,
we will occasionally consider $\cJ$-space monoids, that is,
associative but not necessarily commutative monoid objects in
$(\cS^{\cJ},\boxtimes, U^{\cJ})$.
\begin{example}\label{ex:free-J-space}
  Evaluating a $\cJ$-space at the object $(\bld{d_1},\bld{d_2})$ of $\cJ$
  defines a functor $\mathrm{Ev}^{\cJ}_{(\bld{d_1},\bld{d_2})}\colon
  \cS^{\cJ} \to \cS$. It is right adjoint to the free functor
  $F^{\cJ}_{(\bld{d_1},\bld{d_2})} \colon \cS \to \cS^{\cJ}$ given by
  $F^{\cJ}_{(\bld{d_1},\bld{d_2})}(K) = \cJ((\bld{d_1},\bld{d_2}), - )
  \times K$, \emph{the free $\cJ$-space on $K$ in bidegree}
  $(\bld{d_1},\bld{d_2})$. The values of the free functors for varying
  $(\bld{d_1},\bld{d_2})$ are important examples of $\cJ$-spaces.  We
  note that $0$-simplices $x \in X(\bld{d_1},\bld{d_2})$ correspond to
  $\cJ$-space maps $\bar x\colon F^{\cJ}_{(\bld{d_1},\bld{d_2})}(*)
  \to X$ from the free $\cJ$-space on a point in bidegree
  $(\bld{d_1},\bld{d_2})$.

  If $M$ is a commutative $\cJ$-space monoid and $x \in
  M(\bld{d_1},\bld{d_2})_0$ is a $0$-simplex, then $x$ determines a map of
  commutative $\cJ$-space monoids
  \[ \C{\bld{d_1},\bld{d_2}} = \coprod_{k\geq 0}
    F^{\cJ}_{(\bld{d_1},\bld{d_2})}(*)^{\boxtimes
      k}/\Sigma_k \to M.\] The object $\C{\bld{d_1},\bld{d_2}}$ defined here is
  the \emph{free commutative $\cJ$-space monoid} on a generator in
  bidegree~$(\bld{d_1},\bld{d_2})$. It will often be convenient to use
  the notation $\C{x}$ for $\C{\bld{d_1},\bld{d_2}}$ when discussing
  that map.
\end{example}
The point of defining the category $\cJ$ in this way is the
following interplay with symmetric spectra:
\begin{lemma}\label{lem:adj-J-spaces-SpSym}
There are two adjoint pairs of functors
\begin{equation}\label{eq:adj-J-spaces-SpSym}
  \bS^{\cJ}\colon \cS^{\cJ} \rightleftarrows \Spsym\colon \Omega^{\cJ} \qquad \text{ and } \qquad \bS^{\cJ}\colon \cC \cS^{\cJ} \rightleftarrows \cC\Spsym\colon \Omega^{\cJ}.
\end{equation}
The functor $\bS^{\cJ}\colon (\cS^{\cJ},\boxtimes, U^{\cJ}) \to
(\Spsym, \sm, \bS)$ is strong symmetric monoidal. \qed
\end{lemma}
For symmetric spectra $E$ and $\cJ$-spaces $X$, these functors are given by 
\begin{equation}\label{eq:OmegaJ-and-SJ-explicitly}   
  \Omega^{\cJ}(E)(\bld{n_1},\bld{n_2}) = \Omega^{n_2}E_{n_1}\qquad \text{ and } \qquad 
  \bS^{\cJ}[X]_n = \bigvee_{k\geq 0}   X(\bld{n},\bld{k})_+ \sm_{\Sigma_k} S^k.
\end{equation}
In particular, the lemma states that every commutative symmetric ring
spectrum~$A$ gives rise to a commutative $\cJ$-space monoid
$\Omega^{\cJ}(A)$. Below we indicate why $\Omega^{\cJ}(A)$ may be
viewed as the underlying graded multiplicative $E_{\infty}$ space of
$A$.

To each $\cJ$-space $X$ we can associate the space
\[
X_{h\cJ} = \hocolim_{\cJ} X = \mathrm{diag}\left( [s] \mapsto
  \displaystyle\coprod_{\bld{k_0} \ot \dots \ot \bld{k_s}}
  X(\bld{k_s})\right)
\]
given by its Bousfield--Kan homotopy colimit. A map $X \to Y$ of
$\cJ$-spaces is defined to be a \emph{$\cJ$-equivalence} if the
induced map $X_{h\cJ} \to Y_{h\cJ}$ is a weak homotopy equivalence.
The $\cJ$-equivalences are the weak equivalences in a cofibrantly
generated proper simplicial \emph{positive projective $\cJ$-model
  structure} on $\SJ$, where the fibrant objects are the $\cJ$-spaces
$X$ such that each morphism $(\bm_1, \bm_2) \to (\bn_1, \bn_2)$ in
$\cJ$ with $m_1>0$ induces a weak homotopy equivalence $X(\bm_1,
\bm_2) \to X(\bn_1, \bn_2)$ between (Kan) fibrant simplicial
sets, see~\cite[Proposition 4.8]{Sagave-S_diagram}.

A map $M \to N$ of commutative $\cJ$-space monoids is defined to be a
\emph{$\cJ$-equivalence} if the underlying map of $\cJ$-spaces is a
$\cJ$-equivalence.  These are the weak equivalences in a cofibrantly
generated proper simplicial \emph{positive projective $\cJ$-model
  structure} on $\CSJ$, where the fibrant objects are the commutative
$\cJ$-space monoids whose underlying $\cJ$-spaces are
fibrant~\cite[Proposition 4.10]{Sagave-S_diagram}. In the sequel, we
will refer to this model structure as the \emph{positive $\cJ$-model
structure}, and the notions of cofibrant or fibrant objects in $\CSJ$ or
of cofibrations or fibrations in $\CSJ$ will refer to this model
structure unless otherwise stated.  By construction of the generating
cofibrations for $\CSJ$ in~\cite[Proposition 9.3]{Sagave-S_diagram},
the free commutative $\cJ$-space monoids $\C{\bld{d_1},\bld{d_2}}$
with $d_1>0$ are examples of cofibrant objects in $\CSJ$.

\begin{lemma}[{\cite[Proposition 4.23]{Sagave-S_diagram}}]
\label{Qlem:adj-J-spaces-SpSym}
The adjunctions~\eqref{eq:adj-J-spaces-SpSym} are Quillen adjunctions
with respect to the positive  $\cJ$-model structures on
$\cS^{\cJ}$ and $\cC\cS^{\cJ}$ and the positive stable model
structures on $\Spsym$ and $\cC\Spsym$, respectively.\qed
\end{lemma}

The functor $(-)_{h\cJ}\colon (\cS^{\cJ},\boxtimes, U^{\cJ}) \to (\cS,
\times, *)$ is lax monoidal (but not lax \emph{symmetric}
monoidal), with monoidal structure map $X_{h\cJ}\times Y_{h\cJ} \to
(X\boxtimes Y)_{h\cJ}$  induced by the natural transformation of
$\cJ\times\cJ$-diagrams  \[X(\bld{m_1},\bld{m_2})\times
Y(\bld{n_1},\bld{n_2}) \to (X\boxtimes
Y)((\bld{m_1},\bld{m_2})\concat(\bld{n_1},\bld{n_2}))\] and the functor
$ - \concat -\colon \cJ \times \cJ \to \cJ$. Therefore, the space
$M_{h\cJ}$ associated with a $\cJ$-space monoid $M$ is a simplicial
monoid. If $M$ is commutative, then one can use the fact that $\cJ$ is
a permutative category to show that $M_{h\cJ}$ is an $E_{\infty}$
space over the Barratt--Eccles operad.  (A closely related statement is
proven in~\cite[Proposition~6.5]{Schlichtkrull_Thom-symmetric}.)

This observation can be extended to an operadic description of
$\cC\cS^{\cJ}$: By~\cite[Theorem 1.7]{Sagave-S_diagram}, the category
$\cC\cS^{\cJ}$ is Quillen equivalent to the category of $E_{\infty}$
spaces over~$B\cJ$. So commutative $\cJ$-space monoids correspond 
to $E_{\infty}$ spaces over the underlying additive $E_{\infty}$ space
$QS^0 \simeq B\cJ$ of the sphere spectrum, just as $\bZ$-graded
monoids in algebra can be defined as monoids over the additive monoid
$(\bZ,+)$ of the integers. This is one reason why commutative
$\cJ$-space monoids may be viewed as \emph{$QS^0$-graded $E_{\infty}$
  spaces}. Consequently, we interpret the commutative $\cJ$-space
monoid $\Omega^{\cJ}(A)$ associated with a commutative symmetric ring
spectrum $A$ as the underlying graded $E_{\infty}$ space of $A$. This
point of view is supported by the fact that the underlying graded
multiplicative monoid of $\pi_*(A)$ can be recovered from
$\Omega^{\cJ}(A)$, cf.~\cite[Proposition 4.24]{Sagave-S_diagram}.

Since $\cS^{\cJ}$ is a monoidal model category with respect to the
$\boxtimes$-product, we know that $X \boxtimes Y$ is homotopically
well-behaved if both $X$ and $Y$ are cofibrant.  It is often useful
that this holds under a weaker cofibrancy condition. To state it, we
recall that for an object $(\bld{n_1},\bld{n_2})$ in $\cJ$, the
$(\bld{n_1},\bld{n_2})$-th \emph{latching space}
\[ L_{(\bld{n_1},\bld{n_2})}X = \colim_{(\bld{m_1},\bld{m_2}) \to
  (\bld{n_1},\bld{n_2})} X(\bld{m_1},\bld{m_2}) \] is the colimit over
the full subcategory of the comma category $(\cJ \downarrow
(\bld{n_1},\bld{n_2}))$ generated by the objects that are not
isomorphisms. A $\cJ$-space $X$ is \emph{flat} if the canonical map $
L_{(\bld{n_1},\bld{n_2})}X \to X(\bld{n_1},\bld{n_2})$ is a
cofibration of simplicial sets for each object
$(\bld{n_1},\bld{n_2})$. A commutative $\cJ$-space monoid is flat if
its underlying $\cJ$-space is.
\begin{lemma}\label{lem:properties-of-flat}\begin{enumerate}[(i)]
\item The functor $-\boxtimes Y$ preserves $\cJ$-equivalences if $Y$ is flat. 
\item A $\cJ$-space that is cofibrant in the positive  $\cJ$-model structure is flat. 
\item Cofibrant commutative $\cJ$-space monoids are flat. 
\end{enumerate}
\end{lemma}
\begin{proof}
  This is proven in~\cite[Propositions 8.2, 6.20 and 4.28]{Sagave-S_diagram}.
\qed \end{proof}

\section{The cyclic and replete bar constructions}\label{sec:cy-and-rep}
In this section we introduce the cyclic and replete bar constructions
of commutative $\cJ$-space monoids and recall the definition of the
topological Hochschild homology of symmetric ring spectra. These are
building blocks of the logarithmic topological Hochschild homology to
be defined in Section~\ref{sec:log-thh}.

\subsection{The cyclic bar construction} As usual $\Delta$ denotes the
category with objects $[n] = \{ 0 < \dots < n\}$ for $n \geq 0$, and
order-preserving maps. The category $\Delta$ is a subcategory of
Connes' cyclic category $\Lambda$,
cf.~\cite[Definition~6.1.1]{Loday_cyclic-homology}. The latter has the
same objects as $\Delta$, and additional morphisms $\tau_n \colon [n]
\to [n]$ satisfying $\tau_n^{n+1} = \mathrm{id}$ as well as $\tau_n
\delta_i = \delta_{i-1} \tau_{n-1}$ and $\tau_n\sigma_i =
\sigma_{i-1}\tau_{n+1}$ for $1\leq i \leq n$. The induced
simplicial and cyclic operators are denoted $d_i = \delta_i^*$, $s_i =
\sigma_i^*$ and $t_n = \tau_n^*$, respectively.

\begin{definition}\label{def:cy-bar-constr}
  Let $(M, \mu, \eta)$ be a not necessarily commutative $\cJ$-space
  monoid, and let $X$ be an $M$-bimodule, i.e., a $\cJ$-space with
  commuting left and right $M$-actions.  The \emph{cyclic bar
    construction} $B_{\bullet}^{\cy}(M, X)$ is the simplicial
  $\cJ$-space
  \[ [n] \longmapsto X \boxtimes M^{\boxtimes n} = X \boxtimes M
  \boxtimes \dots \boxtimes M
  \]
  with $n$ copies of $M$.  The $0$-th face map $d_0$ uses the right
  action $X \boxtimes M \to X$, the $i$-th face map $d_i$ for $0<i<n$
  uses the multiplication $\mu \: M \boxtimes M \to M$ of the $i$-th and $(i+1)$-th factors, 
  and the $n$-th face map $d_n$ uses the symmetric structure
  \begin{equation}\label{eq:sym-str-in-Bcy}
    \tau \: (X \boxtimes M^{\boxtimes n-1}) \boxtimes M
    \xrightarrow{\iso} M \boxtimes (X \boxtimes M^{\boxtimes n-1})
  \end{equation}
  followed by the left action $M \boxtimes X \to X$.  The degeneracy
  map $s_i$ inserts the unit $\eta \: U^{\cJ} \to M$ after the $i$-th factor.  

  If $M$ is commutative, we say that an $M$-bimodule $X$ is
  \emph{symmetric} if the right action on $X$ equals the twist
  followed by the left action. In this case, there is an
  \emph{augmentation} $\epsilon \: B^{\cy}_{\bullet}(M, X) \to X$,
  where the codomain is viewed as a constant simplicial object. It is
  given in simplicial degree~$n$ by the $n$-fold (right) action $X
  \boxtimes M^{\boxtimes n} \to X$ and restricts to the identity on
  the $0$-simplices of $B^{\cy}_{\bullet}(M, X)$.

  In the special case when $X = M$, with left and right actions
  given by the multiplication, we write $B^{\cy}_{\bullet}(M) =
  B^{\cy}_{\bullet}(M, M)$.  This is a cyclic $\cJ$-space, with cyclic
  operator $t_n$ given by the symmetric structure as
  in~\eqref{eq:sym-str-in-Bcy}. When $M$ is commutative,
  $B^{\cy}_{\bullet}(M)$ is a cyclic commutative $\cJ$-space monoid
  and the augmentation $\epsilon \: B^{\cy}_{\bullet}(M) \to M$ is a
  cyclic map to the constant cyclic object $M$.
\end{definition}
Applying the diagonal functor from bisimplicial to simplicial sets
object-wise defines a realization functor $|-|$ from simplicial objects in
$\cC\cS^{\cJ}$ to $\cC\cS^{\cJ}$.
\begin{definition}
  The \emph{cyclic bar construction} $B^{\cy}(M,X)$
  (resp.~$B^{\cy}(M)$) is the realization of $B_{\bullet}^{\cy}(M,X)$
  (resp.~$B^{\cy}_{\bullet}(M)$).
\end{definition}
When $M$ is commutative and $X$ is an $M$-module, it follows from the
definition that $B^{\cy}(M,X)$ is a $B^{\cy}(M)$-module.

The realization functor from simplicial objects in $\cS^{\cJ}$ to
$\cS^{\cJ}$ sends degree-wise $\cJ$-equivalences to $\cJ$-equivalences
(this follows from~\cite[Corollary
18.7.5]{Hirschhorn_model}). By Lemma~\ref{lem:properties-of-flat},
$B^{\cy}(M,X)$ captures a well-defined homotopy type as soon as $M$ is
flat.

The cyclic bar construction admits a different description: The category
of commutative $\cJ$-space monoids is tensored over unbased simplicial sets
by setting 
\[ M \tensor K = \left| [n] \mapsto M^{\boxtimes K_{n}} \right|. \]
This uses that the $\boxtimes$-product is the coproduct in
$\cC\cS^{\cJ}$. The multiplication and unit of $M$ give the simplicial
structure maps. This tensor is part of the structure that makes
$\cC\cS^{\cJ}$ a simplicial model category (as defined for example 
in~\cite[Definition 9.1.6]{Hirschhorn_model}). The compatibility with
the model structure lifts from $\cS^{\cJ}$ because the cotensor is the
same for $\cS^{\cJ}$ and $\cC\cS^{\cJ}$~\cite[Proposition
9.9]{Sagave-S_diagram}.  Using $\Delta[1]/\partial\Delta[1]$ as a
model for $S^1$, we obtain:
\begin{lemma}\label{lem:Bcy-as-tensor}
  There is a natural isomorphism $B^{\cy}(M) \cong M \otimes S^1$ in
  $\cC\cS^{\cJ}$. The augmentation $\epsilon \: B^{\cy}(M) \to M$
  corresponds to the collapse map $S^1 \to {*}$.
\end{lemma}
\begin{proof}
  For $0 \leq k \leq n+1$, we let $a_{n,k}\colon [n] \to [1]$ be the
  $n$-simplex of $\Delta[1]$ with $a_{n,k}(i) = 0$ for $i < k$ and
  $a_{n,k}(i) = 1$ if $i \geq k$.  Passing to the quotient $S^1 =
  \Delta[1]/\partial\Delta[1]$ identifies the constant maps $a_{n,0}$
  and $a_{n,n+1}$ and gives an isomorphism $S^1_n \iso \{a_{n,0},
  \dots, a_{n,n}\}$.  The indicated ordering of $S^1_n$ induces an
  isomorphism $M^{\boxtimes S^1_n} \xrightarrow{\iso} M^{\boxtimes
    (1+n)} = B^{\cy}_{n}(M)$. One can check that this is an isomorphism
  of simplicial objects. For example, $\delta_2\colon [1] \to [2]$
  induces $d_2(a_{2,0}) = a_{1,0}$, $d_2(a_{2,1}) = a_{1,1}$,
  and $d_2(a_{2,2}) = a_{1,2}=a_{1,0}$.  Hence $d_2\colon
  M^{\boxtimes S^1_2} \to M^{\boxtimes S^1_1}$ coincides with
  $d_2\colon B^{\cy}_{2}(M)\to B^{\cy}_{1}(M)$ under the specified
  isomorphism.
\qed \end{proof}
\begin{remark}
  The previous description of $B^{\cy}(M)$ also reflects its cyclic
  structure: As explained for example
  in~\cite[7.1.2]{Loday_cyclic-homology},
  $S^1=\Delta[1]/\partial\Delta[1]$ extends to a cyclic set. Using this
  identification, it is easy to see that $[n]\mapsto M^{\boxtimes S^1_n}$ and
  $B^{\cy}_{\bullet}(M)$ are isomorphic as cyclic objects in
  $\cC\cS^{\cJ}$.
\end{remark}

\subsection{Topological Hochschild homology}
Let $A$ be a commutative symmetric ring spectrum. Implementing the
cyclic bar construction in the context of symmetric spectra provides a
cyclic commutative symmetric ring spectrum $B^{\cy}_{\bullet}(A) = \{[n] \mapsto
A^{\wedge(1+n)}\}$, with cyclic structure maps given as in
Definition~\ref{def:cy-bar-constr}. 

\begin{definition}\label{def:THH}
Let $A$ be a cofibrant commutative symmetric ring spectrum. Then we write 
$\THH_{\bullet}(A)=B^{\cy}_{\bullet}(A)$, and define the \emph{topological Hochschild homology} $\THH(A)$ to be the realization of this cyclic object.
\end{definition}
In this definition, the term ``realization'' can have two different meanings, both of which will be relevant for us. On the one hand, applying the diagonal functor from bisimplicial based sets to simplicial based sets in each spectrum degree of $\THH_{\bullet}(A)$ we get a realization internal to $\cC\Spsym$.  On the other hand, we may first form the geometric realization of the smash powers $A^{\wedge(1+n)}$ to get a cyclic object $[n]\mapsto |A^{\wedge(1+n)}|$ in the category of symmetric spectra of topological spaces. The geometric realization of this cyclic object is then a commutative symmetric ring spectrum of topological spaces that comes equipped with an action of the circle group. It will always be clear from the context (or not important) whether we view the realization $\THH(A)$ as a symmetric spectrum internal to simplicial sets or topological spaces.  

\begin{remark}
The reason for the cofibrancy condition in Definition~\ref{def:THH} is that we want $\THH$ to be a homotopy invariant construction. Since the coproduct of cofibrant objects in a general model category is homotopy invariant, and the realization of simplicial objects in symmetric spectra sends degreewise stable equivalences to stable equivalences, a stable equivalence $A\to B$ of cofibrant commutative symmetric ring spectra induces a stable equivalence $\THH(A)\to \THH(B)$. For a commutative symmetric ring spectrum $A$ that is not cofibrant, one should first choose a cofibrant replacement $A^{\mathrm{cof}}\xrightarrow{\simeq} A$, and then apply the cyclic bar construction to~$A^{\mathrm{cof}}$.
\end{remark}

Using the tensor structure of commutative symmetric ring spectra we
can identify $\THH(A)$ with $A \tensor S^1$, the tensor
of $A$ with the simplicial set $S^1 = \Delta[1]/\partial\Delta[1]$, in analogy with Lemma~\ref{lem:Bcy-as-tensor}.

We noted in Lemmas~\ref{lem:adj-J-spaces-SpSym} and  \ref{Qlem:adj-J-spaces-SpSym} 
that $\bS^{\cJ}\colon\cS^{\cJ}\to\Spsym$ is strong symmetric monoidal and that the induced functor of commutative monoids is a left Quillen functor. This immediately gives the next proposition.
\begin{proposition}
There is a natural isomorphism $\THH(\bS^{\cJ}[M])
\cong \bS^{\cJ}[B^{\cy}(M)]$ for each cofibrant commutative $\cJ$-space
monoid $M$.
\end{proposition}

\subsection{The replete bar construction}\label{subsec:replete-bar}
We now discuss an extension of the cyclic bar construction of a
commutative $\cJ$-space monoid that will play a role in our
definition of logarithmic $\THH$ in Section~\ref{sec:log-thh}.
\begin{definition}
  A (commutative) $\cJ$-space monoid $M$ is \emph{grouplike} if the
  simplicial monoid $M_{h\cJ}$ is grouplike.
\end{definition}  
We recall from~\cite[\S 5]{Sagave_spectra-of-units} that the usual
group completion of homotopy commutative simplicial monoids lifts to
commutative $\cJ$-space monoids. To formulate this, we use that a
commutative $\cJ$-space monoid $M$ gives rise to an associative
simplicial monoid $M_{h\cJ}$, and write $B(M_{h\cJ}) =
B(*,M_{h\cJ},*)$ for the usual bar construction of $M_{h\cJ}$
with respect to the cartesian product.

\begin{proposition}[{\cite[Theorem 1.6]{Sagave_spectra-of-units}}]\label{prop:group-compl-model-str}
  The category $\cC\cS^{\cJ}$ admits a \emph{group completion model
    structure}. The cofibrations are the same as in the positive
   $\cJ$-model structure, and $M \to N$ is a weak
  equivalence if and only if the induced map $B(M_{h\cJ}) \to B(N_{h\cJ})$ is a
  weak equivalence. An object is fibrant if and only if it is
  grouplike and positive $\cJ$-fibrant.\qed
\end{proposition}

An important consequence of the group completion model structure is
that its fibrant replacement provides a functorial group completion
$\eta_M \colon M \cof M^{\gp}$ for commutative $\cJ$-space monoids:
The commutative $\cJ$-space monoid $M^{\gp}$ is grouplike, and
$\eta_M$ induces a group completion $M_{h\cJ} \to (M^{\gp})_{h\cJ}$ of
$E_{\infty}$ spaces in the usual sense. We emphasize that the map $\eta_M$  is assumed to be a cofibration, so that $M^{\gp}$ is automatically cofibrant if $M$ is. 

\begin{example}[{\cite[Example
  5.8]{Sagave_spectra-of-units}}]
  Let $\C{\bld{d_1},\bld{d_2}}$ be the free commutative $\cJ$-space
  monoid on a generator in bidegree $(\bld{d_1},\bld{d_2})$ with $d_1
  > 0$, as defined in Example~\ref{ex:free-J-space}.  The map
  $\C{\bld{d_1},\bld{d_2}}_{h\cJ} \to
  (\C{\bld{d_1},\bld{d_2}}^{\gp})_{h\cJ}$ is weakly equivalent to the
  usual group completion map of $E_{\infty}$ spaces $\coprod_{k\geq
    0}B\Sigma_k \to QS^0$.
\end{example}

\begin{construction}\label{constr:rep-bar}
Let $M$ be a commutative $\cJ$-space monoid and let 
\[ \xymatrix{M \ar@{ >->}[r]^{\sim} & M' \ar@{->>}[r] & M^{\gp}}\] be
a functorial factorization of the group completion map $\eta_M$
into an acyclic cofibration followed by a fibration, in the positive
$\cJ$-model structure. The natural augmentation from the cyclic bar
construction to the constant cyclic object functor induces a
commutative diagram of cyclic objects
\[\xymatrix@-1pc{
  B^{\cy}_{\bullet}(M) \ar[r] \ar[d] & B^{\rep}_{\bullet}(M)\ar[r]  \ar[d] &  B^{\cy}_{\bullet}(M^{\gp}) \ar[d] \\
  M \ar[r]& M' \ar[r]& M^{\gp},
}
\]
where $B^{\rep}_{\bullet}(M)$ is defined as the pullback of $  M' \to M^{\gp} \ot B^{\cy}_{\bullet}(M^{\gp})$
and the map $B^{\cy}_{\bullet}(M) \to  B^{\rep}_{\bullet}(M)$ is given by the universal property
of the pullback. 
\end{construction}
\begin{definition}
  Let $M$ be a commutative $\cJ$-space monoid.  The \emph{replete bar
    construction} $B^{\rep}(M)$ is the realization of the cyclic
  object $B^{\rep}_{\bullet}(M)$, and the induced map $\rho \colon
  B^{\cy}(M) \to B^{\rep}(M)$ is called the \emph{repletion map}.
\end{definition}
By definition, the replete bar construction $B^{\rep}(M)$ is a
functorial model for the homotopy pullback of $ M \to M^{\gp} \ot
B^{\cy}(M^{\gp})$. The $\cI$-space version of this definition was
considered in~\cite[Definition~8.10]{Rognes_TLS}.

The fact that $M\to M^{\gp}$ is a $\cJ$-equivalence if $M$ is
grouplike implies the following statement.
\begin{lemma}\label{lem:Bcy--Brep-equiv-for-gplike}
  The repletion map $\rho\colon B^{\cy}(M) \to B^{\rep}(M)$ is a
  $\cJ$-equivalence if $M$ is a grouplike cofibrant commutative
  $\cJ$-space monoid.\qed
\end{lemma}

\subsection{General repletion}
We now introduce a more general notion of \emph{repletion}, which can be
viewed as a relative version of the group completion. Repleteness is a
topological adaption of the algebraic notion of an \emph{exact
  homomorphism} of integral
monoids~\cite[Definition~4.6]{Kato_logarithmic-structures},
compare~\cite[Definition~3.6]{Rognes_TLS}.

\begin{definition}\label{def:repletion}
  Let $\epsilon \: N \to M$ be a map of commutative $\cJ$-space
  monoids. The \emph{repletion} $N^{\rep}\to M$ of $N$ over $M$ is
  defined by factoring $\epsilon$ in the group completion model
  structure as an acyclic cofibration followed by a fibration:
  \[\xymatrix{N\, \ar@{ >->}^-{\sim}[r] & N^{\rep}
    \ar@{->>}[r]& M}. \]
  We write $\rho_N\colon N \to N^{\rep}$ for the \emph{repletion map}, defined by the factorization. 
\end{definition}

Since the group completion model structure is a left Bousfield
localization of the positive  $\cJ$-model structure, it
follows from~\cite[Proposition 3.3.5]{Hirschhorn_model} that
$N^{\rep}$ is well-defined up to $\cJ$-equivalence under $N$ and over
$M$. Repletion relative to the terminal object in $\cC\cS^{\cJ}$ is group completion.

The replete bar construction introduced above can be viewed as a
special case of the general repletion:
\begin{proposition}\label{prop:replete-bar-as-repletion}
There is a chain of $\cJ$-equivalences under $B^{\cy}(M)$ and over $M'$
connecting the replete bar construction $B^{\rep}(M)$ to the repletion $B^{\cy}(M)^{\rep}$
of the augmentation $B^{\cy}(M) \to M$. 
\end{proposition}
We prove the proposition at the end of this section. The reason why we
do not simply define the replete bar construction in terms of the
general repletion is that Construction~\ref{constr:rep-bar} provides a
cyclic object $B^{\rep}_{\bullet}(M)$ with realization $B^{\rep}(M)$.
This extra structure on $B^{\rep}(M)$ is not visible on
$B^{\cy}(M)^{\rep}$. The general notion of repletion is nonetheless
useful, for example for the proofs in
Section~\ref{subsec:logTHH-log}.

In general, fibrations and acyclic cofibrations in a left Bousfield
localization such as the group completion model structure are
difficult to understand. However, we can give a simpler description of
the repletion of maps that are \emph{virtually surjective}, in the
sense of the following $\cJ$-space variant of~\cite[Definition 8.1]{Rognes_TLS}.  
\begin{definition}
  A map $\epsilon \colon N \to M$ of commutative $\cJ$-space monoids
  is \emph{virtually surjective} if it induces a surjective homomorphism of
  abelian groups $\pi_0(N^{\gp})_{h\cJ}\to \pi_0(M^{\gp})_{h\cJ}$.
\end{definition}

\begin{lemma}\label{lem:repletion-as-pullback}
  Let $\epsilon \: N \to M$ be a virtually surjective map of
  commutative $\cJ$-space monoids, and consider the diagram of solid
  arrows
\begin{equation}
\xymatrix{
N \ar[rr]^{\rho_N} \ar[d]_{\eta_N} && N^{\rep} \ar@{-->}[d] \ar[rr] && M \ar[d]^{\eta_M} \\
N^{\gp} \ar@{ >->}[rr]^-{\sim} && (N^{\gp})' \ar@{->>}[rr] && M^{\gp},
}
\end{equation}
where the bottom row is a factorization in the positive $\cJ$-model
structure. Then there exists a map $N^{\rep} \to (N^{\gp})'$ such that
the diagram commutes, and for any such map the right hand square is
homotopy cartesian with respect to the positive $\cJ$-model structure.
\end{lemma}
\begin{proof}
  The map $(N^{\gp})' \to M^{\gp}$ is a fibration in the group
  completion model structure by~\cite[Proposition
  3.3.16]{Hirschhorn_model}. Hence the lifting axioms in the group
  completion model structure provide the desired map $N^{\rep} \to
  (N^{\gp})'$.  The base change of $(N^{\gp})' \to M^{\gp}$ along
  $\eta_M$ provides a map $N' \to M$ that is also a fibration in the
  group completion model structure. Since $N^{\rep} \to M$ has this
  property by construction, it follows from~\cite[Proposition
  3.3.5]{Hirschhorn_model} that the induced map $N^{\rep}\to N'$ is a
  $\cJ$-equivalence as soon as it is a weak equivalence in the group
  completion model structure. The two out of three axiom for weak
  equivalences reduces to showing that $N' \to (N^{\gp})'$ is a weak
  equivalence in the group completion model structure. We claim that
  an application of the Bousfield--Fried\-lander Theorem~\cite[Theorem
  B.4]{Bousfield-F_Gamma-bisimplicial}, similar to the proof
  of~\cite[Lemma 5.3]{Bousfield-F_Gamma-bisimplicial}, shows that the
  induced square
\[\xymatrix@-1pc{
  B(N'_{h\cJ}) \ar[r] \ar[d] & B(M_{h\cJ}) \ar[d]^{\sim}\\
  B((N^{\gp})'_{h\cJ}) \ar[r] & B((M^{\gp})_{h\cJ}) }\] is homotopy
cartesian. For this we note that the square in question results from a
pointwise homotopy cartesian square of bisimplicial sets, that the
bisimplicial sets $B_{\bullet}((N^{\gp})'_{h\cJ})$ and
$B_{\bullet}((M^{\gp})_{h\cJ})$ satisfy the $\pi_*$-Kan condition
because the simplicial monoids $(N^{\gp})'_{h\cJ} $ and $
(M^{\gp})_{h\cJ}$ are grouplike, and that the virtual surjectivity of
$\epsilon$ implies that $B_{\bullet}((N^{\gp})'_{h\cJ})\to
B_{\bullet}((M^{\gp})_{h\cJ})$ induces a Kan fibration on vertical
path components.  Hence~\cite[Theorem
B.4]{Bousfield-F_Gamma-bisimplicial} applies, and the claim of the
lemma follows.
\qed \end{proof}
The next corollary relates the repletion defined here to the
$\cJ$-space version of the notion used in~\cite[\S 8]{Rognes_TLS},
compare also the discussion in~\cite[\S 5.10]{Sagave-S_group-compl}.
\begin{corollary}\label{cor:repletion-as-hty-pb}
 Let $\epsilon \: N \to M$ be a virtually surjective map in $\cC\cS^{\cJ}$. Then the repletion $N^{\rep}$ is 
$\cJ$-equivalent to the homotopy pullback of $N^{\gp} \to M^{\gp} \ot M$ with respect to the positive $\cJ$-model structure. \qed
\end{corollary}
We now return to the cyclic bar construction and prepare for the proof
of Proposition~\ref{prop:replete-bar-as-repletion}.
\begin{lemma}\label{lem:BcyMgp-vs-BcyM-gp} The commutative $\cJ$-space monoids 
  $B^{\cy}(M^{\gp})$ and $B^{\cy}(M)^{\gp}$ are $\cJ$-equivalent as
  commutative $\cJ$-space monoids under $B^{\cy}(M)$ and over $M^\gp$.
\end{lemma}
\begin{proof}
  As observed in Lemma~\ref{lem:Bcy-as-tensor}, there is an
  isomorphism $B^{\cy}(N) \iso N \tensor S^1$.  The group completion
  and the collapse map $S^1 \to *$ induce the outer square in the
  commutative diagram
\[\xymatrix@-1pc{
  M \tensor S^1 \ar[dd] \ar[r] & (M \tensor S^1)^{\gp} \ar@{ >->}[d]^{\sim} \\
  & (M\tensor S^1)' \ar@{->>}[d]\\
  M^{\gp} \tensor S^1 \ar[r] \ar@{-->}[ur] & M^{\gp}.  }\] Since the
positive  $\cJ$-model structure is simplicial, it follows
from~\cite[Theorem 4.1.1 (4)]{Hirschhorn_model} that the group
completion model structure is also simplicial. So the left hand
vertical map is an acyclic cofibration in the group completion model
structure.  The object $(M\tensor S^1)'$ is defined by forming the indicated
factorization in the group completion model structure. Then $(M\tensor
S^1)'$ is also grouplike.  The model category axioms in
the group completion model structure provide the lift $M^{\gp} \tensor
S^1 \to (M \tensor S^1)'$. The two out of three axiom implies that the
lift is a weak equivalence in the group completion model structure. To
see that it is a $\cJ$-equivalence, it is enough to show that $
M^{\gp} \tensor S^1$ is also grouplike. For this we note that the
monoids of zero-simplices of $(M^{\gp})_{h\cJ}$ and $(M^{\gp} \tensor
S^1)_{h\cJ} $ coincide, since they are both given by disjoint union of the sets
of zero-simplices of $M^{\gp}(\bld{m_1},\bld{m_2})$ over all objects
$(\bld{m_1},\bld{m_2})$ of $\cJ$. If two $0$-simplices of $(M^{\gp})_{h\cJ}$
become equivalent in $\pi_0((M^{\gp})_{h\cJ})$, they also become equivalent in
$\pi_0((M^{\gp} \tensor S^1)_{h\cJ})$. Hence the latter monoid is a group if the former
one is. 
\qed \end{proof}
\begin{proof}[ 
of Proposition~\ref{prop:replete-bar-as-repletion}]
  By the previous lemma, $B^{\rep}(M)$ is $\cJ$-equivalent to the
  homotopy pullback of $M \to M^{\gp} \ot B^{\cy}(M)^{\gp}$. Since
  $B^{\cy}(M)\to M$ has a multiplicative section, it is virtually
  surjective, and so it follows from
  Corollary~\ref{cor:repletion-as-hty-pb} that the homotopy pullback
  of $M \to M^{\gp} \ot B^{\cy}(M)^{\gp}$ is $\cJ$-equivalent to the
  repletion of the map $B^{\cy}(M)\to M$.
\qed \end{proof}

\section{Logarithmic \texorpdfstring{$\THH$}{THH}}\label{sec:log-thh}
In this section we define pre-log and log (symmetric) ring spectra and introduce
their topological Hochschild homology. 

\begin{definition}\label{def:pre-log-ring-spectrum}
  A \emph{pre-log structure} $(M, \alpha)$ on a commutative symmetric
  ring spectrum $A$ is a commutative $\cJ$-space monoid $M$ and a
  commutative $\cJ$-space monoid map $\alpha \: M \to
  \Omega^{\cJ}(A)$.  A \emph{pre-log ring spectrum} $(A, M,
  \alpha)$ is a commutative symmetric ring spectrum $A$ with a choice
  of pre-log structure $(M, \alpha)$. A morphism $(f,f^{\flat})\colon
  (A,M,\alpha)\to (B,N,\beta)$ is a pair of morphisms $f\colon A \to
  B$ in $\cC\Spsym$ and $f^{\flat}\colon M \to N$ in $\cC\cS^{\cJ}$
  such that $\Omega^{\cJ}(f)\alpha = \beta f^{\flat}$.

  Specifying $\alpha$ is equivalent to specifying its adjoint,
  the commutative symmetric ring spectrum map $\bar\alpha \:
  \bS^{\cJ}[M] \to A$.  We often omit $\alpha$ from the notation.
\end{definition}
As suggested by the terminology, there is also the notion of a log
ring spectrum. It will be defined in Section~\ref{subsec:logification}
below.
\begin{remark}\label{rem:pre-log-terminology}
  Throughout, \emph{log} is short for \emph{logarithmic}.  Our pre-log
  ring spectra were called \emph{graded pre-log symmetric ring
  spectra} in~\cite[\S 4.30]{Sagave-S_diagram}
  and~\cite{Sagave_log-on-k-theory}, to distinguish them from the
  earlier notion of pre-log symmetric ring spectra introduced
  in~\cite{Rognes_TLS}.  When the latter reference was written the
  theory of $\cJ$-spaces was not yet properly developed, so only the
  ``ungraded'' version of $E_\infty$ spaces known as $\cI$-spaces was
  considered.  That restricted theory suffers from a lack of really
  interesting examples for log structures on ring spectra that are not
  Eilenberg--Mac\,Lane spectra, which is alleviated by the passage to
  the more general context of $\cJ$-spaces.  It now seems
  sensible to shift the terminology, so that the most interesting
  objects (commutative symmetric ring spectra with pre-log structures
  given by commutative $\cJ$-space monoids) have the simplest name.
\end{remark}

\begin{example}\label{ex:pre-log-str}\begin{enumerate}[(i)]
  \item Let $M$ be a commutative $\cJ$-space monoid.  The adjunction
    unit $\zeta \: M \to \Omega^\cJ(\bS^{\cJ}[M])$ defines the
    \emph{canonical pre-log structure} $(M, \zeta)$ on $\bS^{\cJ}[M]$,
    with adjoint the identity map of $\bS^{\cJ}[M]$.
  \item Let $A$ be a commutative symmetric ring spectrum. A map
    $x\colon S^{d_2} \to A_{d_1}$ defines a $0$-simplex $x \in
    \Omega^{\cJ}(A)(\bld{d_1},\bld{d_2})_0$. As explained in
    Example~\ref{ex:free-J-space}, the map $x$ induces a map $\C{x}
    \to \Omega^{\cJ}(A)$ from the free commutative $\cJ$-space monoid
    on a point in bidegree $(\bld{d_1},\bld{d_2})$ to
    $\Omega^{\cJ}(A)$. This defines the \emph{free pre-log structure}
    on~$A$ generated by~$x$.
  \item Pre-log rings in the algebraic sense give rise to pre-log ring
    spectra.  We study this in detail in
    Section~\ref{sec:discrete-pre-log-rings}.
\end{enumerate}
\end{example}
The following definition is an important source of interesting pre-log
structures:
\begin{definition}\label{def:direct-image-pre-log} 
  Let $j\colon A \to B$ be a map of commutative symmetric ring spectra
  and let $N \to \Omega^{\cJ}(B)$ be a pre-log structure. The pullback
  of $N \to \Omega^{\cJ}(B) \ot \Omega^{\cJ}(A)$ defines a pre-log
  structure $j_*N = N\times_{ \Omega^{\cJ}(B)}\Omega^{\cJ}(A)$ on $A$
  that we refer to as the \emph{direct image pre-log structure}.
\end{definition}
In order to ensure that the pullback in the definition captures a well-defined homotopy type, we will only consider direct image pre-log
structures when  $j\colon A \to B$ or $N \to \Omega^{\cJ}(B)$ are positive fibrations, and $B$ is positive fibrant. 

Now we turn to the definition of logarithmic topological Hochschild
homology. Our strategy will be to first define it on pre-log ring
spectra satisfying a suitable cofibrancy condition, and then extend the
definition to all pre-log ring spectra by precomposing with a cofibrant
replacement functor.

\begin{definition}\label{def:cof-pre-log-ring}
  A pre-log ring spectrum $(A,M,\alpha)$ is \emph{cofibrant} if $M$ is a
  cofibrant commutative $\cJ$-space monoid and the adjoint structure
  map $\bar\alpha\colon \bS^{\cJ}[M]\to A$ is a cofibration of
  commutative symmetric ring spectra.
\end{definition}
We note that if $(A,M,\alpha)$ is cofibrant, then $A$ is cofibrant as
a commutative symmetric ring spectrum. It follows from standard model
category arguments that the cofibrant pre-log ring spectra are the
cofibrant objects in a cofibrantly generated projective model
structure where a map $(f,f^{\flat})$ is a fibration or a weak
equivalence if and only if both $f$ and $f^{\flat}$ have this property. This
implies that we may choose a cofibrant replacement functor
$(A^{\mathrm{cof}},M^{\mathrm{cof}},\alpha^{\mathrm{cof}}) \to
(A,M,\alpha)$ for pre-log ring spectra. Thus, once we define 
log $\THH$ for cofibrant pre-log ring spectra below, the definition can easily be extended to all pre-log ring spectra by precomposing with this cofibrant replacement functor. Such cofibrant replacements were used implicitly in the formulation of Theorems~\ref{thm:THH-cofiber-sequence-with-repetitive}
and~\ref{thm:THH-cofiber-sequence-with-direct-image} from the
introduction. 
\begin{definition}\label{def:log-thh}
  Let $(A, M, \alpha)$ be a cofibrant  pre-log ring spectrum. Its
  \emph{logarithmic topological Hochschild homology} is the
  commutative symmetric ring spectrum
  \[\THH(A,M) = \THH(A)\wedge_{\bS^{\cJ}[B^{\cy}(M)]}
  \bS^{\cJ}[B^{\rep}(M)]\] given by the pushout in the following diagram 
  \begin{equation}\label{eq:def-log-thh}
    \xymatrix@-1pc{
      \bS^{\cJ}[B^{\cy}(M)] \ar[r]^-{\rho} \ar[d] & 
      \bS^{\cJ}[B^{\rep}(M)] \ar[d] \\
      \THH(A) \ar[r]^-{\rho} & \THH(A,M)
    }
  \end{equation}
  of commutative symmetric ring spectra. The upper horizontal arrow is
  given by applying $\bS^{\cJ}$ to the repletion map $\rho \:
  B^{\cy}(M) \to B^{\rep}(M)$ and the
  left hand vertical map is obtained by applying the functor $\THH$ to the
  adjoint pre-log structure map $\bar\alpha \:
  \bS^{\cJ}[M] \to A$, under the
  identification $\THH(\bS^{\cJ}[M]) \cong
  \bS^{\cJ}[B^{\cy}(M)]$.
\end{definition}
It is clear from the construction that $\THH(A,M)$ is isomorphic to the realization of the cyclic commutative symmetric ring spectrum $\THH_{\bullet}(A,M)$ defined by the pushout of the diagram
  \[
  \THH_{\bullet}(A) \ot \bS^{\cJ}[B^{\cy}_{\bullet}(M)]
  \to \bS^{\cJ}[B^{\rep}_{\bullet}(M)]
  \]
  in cyclic commutative symmetric ring spectra. Hence the geometric realization of $\THH_{\bullet}(A,M)$ becomes a commutative symmetric ring spectrum with circle action, which we shall also denote by $\THH(A,M)$. It will always be clear from the context (or not important) whether we think of $\THH(A,M)$ as a symmetric spectrum of simplicial sets or topological spaces.  

\begin{remark}
The point of the cofibrancy condition on $(A, M, \alpha)$ in 
Definition~\ref{def:log-thh} is that the adjoint structure map $\bar\alpha\colon \bS^{\cJ}[M] \to A$ being a cofibration implies that
$\bS^{\cJ}[B_{\bullet}^{\cy}(M)] \to \THH_{\bullet}(A)$
is a cofibration in every simplicial degree and that the realization 
$\bS^{\cJ}[B^{\cy}(M)] \to \THH(A)$ is a cofibration of
commutative symmetric ring spectra. This ensures that the
pushout squares defining $\THH(A,M)$ and $\THH_{\bullet}(A,M)$ are
homotopy pushout squares. In fact, $\THH(A,M)$ also represents the left derived balanced smash product of $\THH(A)$ and $\bS^{\cJ}[B^{\rep}(M)]$ thought of as $\bS^{\cJ}[B^{\cy}(M)]$-module spectra. This follows by applying the next lemma to the cofibrant pre-log structure on $\THH(A)$ defined by $\THH(\bar\alpha)$.   
\end{remark}

\begin{lemma}\label{lem:ext-along-cof-pre-st-eq}
  Let $(A,M)$ be a cofibrant pre-log ring spectrum. Extension of
  scalars along the adjoint structure map $A\sm_{\bS^{\cJ}[M]}(-)
  \colon \Mod_{\bS^{\cJ}[M]}\to\Mod_{A}$ preserves stable equivalences
  between not necessarily cofibrant objects.
\end{lemma}
\begin{proof}
  We consider more generally a cofibration $E\to F$ of commutative
  symmetric ring spectra. By~\cite[Proposition
  4.1]{Shipley_convenient}, $F$ is cofibrant as a flat $E$-module,
  where our use of the term ``flat'' is synonymous with the term
  ``$E$-cofibrant'' used in~\cite{Shipley_convenient}.  A cell
  induction argument reduces the claim to the following statement: If
  $Y$ is an $E$-module that is obtained from an $E$-module $X$ by
  attaching a generating cofibration of flat $E$-modules $(K \to L)\sm
  E$, then $Y\sm_E(-)$ preserves stable equivalences if $X\sm_E(-)$
  does. Since the smash products with the flat symmetric spectra $K$
  and $L$ preserve stable equivalences and the smash product with any
  symmetric spectrum sends flat cofibrations to level cofibrations,
  the gluing lemma for stable equivalences and level cofibrations of
  symmetric spectra shows the claim.
\qed \end{proof}

\begin{proposition}\label{prop:log-THH-hty-inv}
  If $(f,f^{\flat})\colon (A,M) \to (B,N)$ is a map of cofibrant pre-log ring
  spectra such that $f \colon A \to B$ is a stable equivalence and
  $f^\flat\colon M \to N$ is a $\cJ$-equivalence, then the induced map
  $\THH(A,M) \to \THH(B,N)$ is a stable equivalence.
\end{proposition}
\begin{proof}
  The cofibrancy conditions imply that $f$ gives rise to a stable
  equivalence $\THH(A)\to \THH(B)$ and that $f^{\flat}$ gives rise to
  $\cJ$-equivalences $B^{\cy}(M)\to B^{\cy}(N)$ and $B^{\rep}(M)\to
  B^{\rep}(N)$. Although $B^{\rep}(M)$ and $B^{\rep}(N)$ are not
  necessarily cofibrant, it follows from
  Corollary~\ref{cor:bSJ-preserves-we} that the induced maps
  \[
  \bS^{\cJ}[B^{\cy}(M)]\to \bS^{\cJ}[B^{\cy}(N)],\qquad
  \bS^{\cJ}[B^{\rep}(M)]\to \bS^{\cJ}[B^{\rep}(N)]
  \]
  are stable equivalences. Hence the result follows from left
  properness of the positive stable model structure on $\cC\Spsym$.
\qed \end{proof}
This result implies in particular that we obtain a homotopy invariant
functor if we precompose our log $\THH$ functor with a cofibrant
replacement functor.

\begin{proposition}
For a cofibrant commutative $\cJ$-space monoid $M$, the natural map 
\[
\bS^{\cJ}[B^{\rep}(M)] \xrightarrow{\cong} \THH(\bS^{\cJ}[M], M)
\]
is an isomorphism, and the
canonical map $(\bS^{\cJ}[M],M)\to (A,M)$ induces a natural
pushout square 
\[
\xymatrix@-1pc{
\THH(\bS^{\cJ}[M]) \ar[r]^-{\rho} \ar[d]_-{\bar\alpha} &
\THH(\bS^{\cJ}[M], M) \ar[d] \\
\THH(A) \ar[r]^-{\rho} & \THH(A, M)
}
\]
of commutative symmetric ring spectra.\qed
\end{proposition}

\begin{remark}
If $(f,f^{\flat})\colon (B,N) \to (A,M)$ is a map of pre-log ring spectra, then the
repletion $N^\rep \to M$ extends to a map of pre-log ring spectra
\[ (B \sm_{\bS^{\cJ}[N]}\bS^{\cJ}[N^{\rep}], N^\rep) \to (A,M). \]
We call this map the repletion of $(f,f^{\flat})$. 

The adjoints of the vertical maps in~\eqref{eq:def-log-thh}
define pre-log ring spectra 
%$(\THH(A),B^{\cy}(M))$ and $(\THH(A, M),B^{\rep}(M))$. 
\[(\THH(A),B^{\cy}(M)) \qquad \text{and}\qquad  (\THH(A, M),B^{\rep}(M)).\] 
The augmentation of the
cyclic bar construction induces an augmentation
%$
\[
(\THH(A),B^{\cy}(M)) \to (A,M),
\]
%$ 
and it follows from Proposition~\ref{prop:replete-bar-as-repletion} that the repletion of this map is stably equivalent to $(\THH(A,M),B^{\rep}(M))$. 
\end{remark}

\subsection{Log \texorpdfstring{$\THH$}{THH} and localization}
We now explain how the logarithmic $\THH$ of $(A,M)$ lies between
$\THH$ of $A$ and $\THH$ of the localization of $A$ away from $M$.
\begin{definition}[{\cite[Definition 7.15]{Rognes_TLS}}]
Let $(A,M)$ be a pre-log ring spectrum. The commutative symmetric ring
spectrum given by the pushout 
\[ A[M^{-1}] = A \sm_{\bS^{\cJ}[M]} \bS^{\cJ}[M^{\gp}]\] is the
\emph{localization} of $(A,M)$, and the pre-log ring spectrum
$(A[M^{-1}],M^\gp)$ is the \emph{trivial locus} of $(A,M)$. 
\end{definition}
We note that since $\eta_M\colon M \to M^{\gp}$ is a cofibration, the pre-log ring spectrum $(A[M^{-1}],M^\gp)$ is
cofibrant if $(A,M)$ is. 
\begin{example} The trivial locus of the pre-log ring spectrum
  $(\bS^{\cJ}[M], M)$ from Example~\ref{ex:pre-log-str}(i) is $(\bS^{\cJ}[M^{\gp}],
  M^{\gp})$.
\end{example}
\begin{example}[{\cite[Proposition 3.19]{Sagave_log-on-k-theory}}]
  If the map $x \colon S^{d_2} \to A_{d_1}$ represents a homotopy class
  $[x]\in \pi_{d_2-d_1}(A)$ of a positive fibrant commutative
  symmetric ring spectrum~$A$, then the map $A \to A[\C{x}^{-1}]$ induces the localization
  homomorphism $\pi_*(A) \to (\pi_*(A))[1/[x]]$ at the level of homotopy groups.
\end{example}
\begin{proposition}
  Let $(A, M)$ be a cofibrant pre-log ring spectrum.  Then there is a
  natural factorization of the localization map $\THH(A) \to
  \THH(A[M^{-1}])$ through the
  repletion map $\THH(A) \to \THH(A, M)$.
\end{proposition}
\begin{proof}
  The claim follows by observing that the maps
  $B^{\cy}(M) \to B^{\rep}(M)$ and  $B^{\rep}(M)\to
  B^{\cy}(M^{\gp})$ induce pushout squares
\[
\xymatrix@-1pc{
\bS^{\cJ}[B^{\cy}(M)] \ar[rr] \ar[d]
	&& \bS^{\cJ}[B^{\rep}(M)] \ar[rr] \ar[d]
	&& \bS^{\cJ}[B^{\cy}(M^{\gp})] \ar[d] \\
\THH(A) \ar[rr] && \THH(A, M) \ar[rr]
	&& \THH(A[M^{-1}])
}
\]
of commutative symmetric ring spectra.
\qed \end{proof}

\subsection{Homotopy cofiber sequences for log \texorpdfstring{$\THH$}{THH}}
We shall use the following proposition as a general principle for setting up homotopy cofiber sequences involving log~$\THH$. This will be used to construct localization sequences for discrete rings in 
Section~\ref{sec:discrete-pre-log-rings} and for periodic ring spectra in 
Section~\ref{sec:hty-cof-log-thh}.

\begin{proposition}\label{prop:general-logTHH-homotopy-cofiber}
Let $(A,M)$ be a cofibrant pre-log ring spectrum and suppose that $P$ is a cofibrant commutative $\cJ$-space monoid such that there is a map of commutative symmetric ring spectra $\bS^{\cJ}[M]\to \bS^{\cJ}[P]$ and a homotopy cofiber sequence
\[
\bS^{\cJ}[B^{\cy}(M)]\overset{\rho}{\longto} \bS^{\cJ}[B^{\rep}(M)]\overset{\partial}{\longto} \Sigma\bS^{\cJ}[B^{\cy}(P)] 
\]
of $\bS^{\cJ}[B^{\cy}(M)]$-modules with circle action. Then there is a homotopy cofiber sequence
\[
\THH(A)\overset{\rho}{\longto} \THH(A,M)\overset{\partial}{\longto} \Sigma\THH(A\wedge_{\bS^{\cJ}[M]}\bS^{\cJ}[P])
\] 
of $\THH(A)$-modules with circle action.
\end{proposition}

\begin{proof}
Applying base change along $\bS^{\cJ}[B^{\cy}(M)] \to \THH(A)$ to the
homotopy cofiber sequence of $\bS^{\cJ}[B^{\cy}(M)]$-modules in the
proposition, we get a homotopy cofiber sequence
\[
\THH(A)\to \THH(A,M) \to \Sigma\THH(A)\wedge_{\bS^{\cJ}[B^{\cy}(M)]}\bS^{\cJ}[B^{\cy}(P)]
\]
of $\THH(A)$-modules with circle action. This uses that the functor in question takes stable equivalences of  
$\bS^{\cJ}[B^{\cy}(M)]$-modules to stable equivalences of $\THH(A)$-modules, as follows from  
Lemma~\ref{lem:ext-along-cof-pre-st-eq} applied to the cofibrant pre-log ring spectrum $(\THH(A),\bS^{\cJ}[B^{\cy}(M)])$. To get the statement in the proposition we identify the last term via the sequence of isomorphisms
\[
\begin{gathered}
\THH(A)\wedge_{\bS^{\cJ}[B^{\cy}(M)]}\bS^{\cJ}[B^{\cy}(P)] 
\cong|B^{\cy}_{\bullet}(A)\wedge_{B^{\cy}_{\bullet}(\bS^{\cJ}[M])}B^{\cy}_{\bullet}(\bS^{\cJ}[P])|\\
\cong |B^{\cy}_{\bullet}(A\wedge_{\bS^{\cJ}[M]}\bS^{\cJ}[P])|=\THH(A\wedge_{\bS^{\cJ}[M]}\bS^{\cJ}[P])
\end{gathered}
\]
of $\THH(A)$-modules with circle action. Notice that $P$ being cofibrant ensures that $A\wedge_{\bS^{\cJ}[M]}\bS^{\cJ}[P]$ is a cofibrant commutative symmetric ring spectrum, which justifies the notation $\THH(A\wedge_{\bS^{\cJ}[M]}\bS^{\cJ}[P])$.
\qed \end{proof}

\begin{remark}
Our use of the terminology ``$\THH(A)$-module with circle action'' in Proposition~\ref{prop:general-logTHH-homotopy-cofiber} refers to a module for 
$\THH(A)$ thought of as a commutative monoid in the symmetric monoidal category of symmetric spectra with circle action. Thus, if $X$ denotes a 
$\THH(A)$-module with circle action, it is understood that the action map $\THH(A)\wedge X\to X$ is circle equivariant.
\end{remark}

\subsection{Logification}\label{subsec:logification}
In order to introduce log ring spectra, we recall the definition of units:
\begin{definition}\label{def:graded-units}
  The (graded) units $\GLoneJof(A)$ of a positive fibrant commutative
  symmetric ring spectrum $A$ is the sub commutative $\cJ$-space monoid of
  $\Omega^{\cJ}(A)$ consisting of the
path components that map to units in $\pi_0(
  \Omega^{\cJ}(A)_{h\cJ})$.
\end{definition}
This notion of units is extensively studied in~\cite{Sagave-S_diagram}
and \cite{Sagave_spectra-of-units}. The use of $\cJ$-spaces ensures
that the canonical map $\GLoneJof(A) \to \Omega^{\cJ}(A)$ corresponds
to the inclusion $\pi_*(A)^{\times} \to (\pi_*(A),\cdot)$ of the
units in the underlying signed \emph{graded} multiplicative monoid of the
graded commutative ring of homotopy groups $\pi_*(A)$.

\begin{definition}
A pre-log ring spectrum $(A,M,\alpha)$ is a \emph{log ring spectrum}
if the top horizontal map $\widetilde{\alpha}$ in the following pullback
square is a $\cJ$-equivalence:
\[\xymatrix@-1pc{
\alpha^{-1}\GLoneJof(A) \ar[d] \ar[r]^-{\widetilde{\alpha}} & \GLoneJof(A)\ar[d]\\
M \ar[r]^-{\alpha} & \Omega^{\cJ}(A)
}\]
\end{definition}
The map $\GLoneJof(A) \to \Omega^{\cJ}(A)$ is a fibration because
it is an inclusion of path components. Hence the log condition is a homotopy invariant
notion as soon as $A$ is positive fibrant. 
\begin{example}\label{ex:log-str}\begin{enumerate}[(i)]
  \item The inclusion $\GLoneJof(A) \to \Omega^{\cJ}(A)$ of
    commutative $\cJ$-space monoids defines a log structure, known as
    the \emph{trivial log structure} on~$A$.
  \item If $(B,N)$ is a log ring spectrum, then the direct image
    pre-log structure associated with $A \to B$, introduced in
    Definition~\ref{def:direct-image-pre-log}, is a log structure if $N
    \to \Omega^{\cJ}(B)$ is a fibration in $\cC\cS^{\cJ}$ or $A\to B$
    is a positive fibration in $\cC\Spsym$.
\end{enumerate}
\end{example}
We can combine these two examples in order to produce log structures
from maps in $\cC\Spsym$:
\begin{definition}\label{def:direct-image-log-str}
  Let $j\colon A \to B$ be a map of commutative symmetric ring spectra
  with $B$ positive fibrant. The \emph{direct image log structure}
  $j_*\!\GLoneJof(B) \to \Omega^{\cJ}(A)$ is obtained by forming the
  direct image log structure associated with $j\colon A \to B$ and the
  trivial log structure $\GLoneJof(B) \to \Omega^{\cJ}(B)$.
\end{definition}
As illustrated in
Theorem~\ref{thm:THH-cofiber-sequence-with-direct-image}
and~\cite[Theorem 1.4]{Sagave_log-on-k-theory}, the direct image log
structure is an interesting log structure when $A \to B$ is the
connective cover map of a periodic ring spectrum $B$.
\begin{construction}\label{cons:logification}
  If $(A,M,\alpha)$ is a pre-log ring spectrum, then the
  \emph{associated log ring spectrum} $(A^a,M^a,\alpha^a)$ is defined as
  follows: We choose a factorization
\[\xymatrix{\alpha^{-1}\GLoneJof(A)\, \ar@{ >->}[r] & G \ar@{->>}[r]^-{\sim} & \GLoneJof(A) }\]
of $\widetilde{\alpha}$ in the positive $\cJ$-model structure, and
define $M^a$ by the following pushout square in $\cC\cS^{\cJ}$
\begin{equation}\label{eq:logification-pushout}\xymatrix@-1pc{
    \alpha^{-1}\GLoneJof(A) \ar@{ >->}[r] \ar[d]& G \ar[d]\\
    M \ar[r] & M^a.}
\end{equation}
The universal property of the pushout determines a map $M^a \to
\Omega^{\cJ}(A)$. As proven in~\cite[Lemma 7.7]{Rognes_TLS}
or~\cite[Lemma 3.12]{Sagave_log-on-k-theory}, the induced map $M \to
M^a$ is a $\cJ$-equivalence if $(A,M)$ is already a log ring spectrum,
and $(A,M^a)$ is always a log ring spectrum. However, in connection
with log $\THH$ we want a logification that preserves cofibrancy, so
we carry on and define $A^a$ by the following functorial factorization
\[
\xymatrix@-1pc{
A\wedge_{\bS^{\cJ}[M]}\bS^{\cJ}[M^a] \ar@{>->}[r] \ar[d]& A^a\ar@{->>}[d]^{\simeq}\\
A\wedge_{\bS^{\cJ}[M]}A \ar[r] & A
}
\]
in the positive stable model structure on $\cC\Spsym$. We let $\alpha^a$ be the adjoint of the map $\bS^{\cJ}[M^a]\to A^a$ given by the factorization. 
In this way we have defined a functor $(A,M,\alpha)\mapsto  (A^a,M^a,\alpha^a)$ from the category of pre-log ring spectra to the full subcategory of log ring spectra that comes with a natural morphism $(A,M,\alpha)\to  (A^a,M^a,\alpha^a)$ of pre-log ring spectra such that $A\to A^a$ is a stable equivalence. If $(A,M,\alpha)$ is cofibrant, then also $(A^a,M^a,\alpha^a)$ is cofibrant.
\end{construction}

\begin{example}\label{ex:Dx}
  Let $E$ be a $d$-periodic commutative symmetric ring spectrum with
  connective cover $j\colon e \to E$, where $d>0$. As explained
  in~\cite[Construction 4.2]{Sagave_log-on-k-theory}, the choice of a
  representative $x$ of a periodicity element gives rise to a pre-log
  structure $D(x) \to \Omega^{\cJ}(e)$. It has the property $D(x)_{h\cJ}
  \simeq Q_{\geq
    0}S^0$~\cite[Lemma~4.6]{Sagave_log-on-k-theory}. There is a
  canonical map $(e,D(x))\to (e,j_*\GLoneJof(E))$ that induces a weak
  equivalence $(e^a,D(x)^a)\to (e,j_*\GLoneJof(E))$. We will return to
  these pre-log structures in the sequel~\cite{RSS_LogTHH-II}.
\end{example}

\subsection{Log \texorpdfstring{$\THH$}{THH} and logification}\label{subsec:logTHH-log}
Log $\THH$ is invariant under logification:
\begin{theorem}\label{thm:logification-inv-of-THH}
  Let $(A,M,\alpha)$ be a cofibrant pre-log ring
  spectrum. 
  Then the logification map $(A,M) \to (A^a,M^a)$ induces a stable
  equivalence $\THH(A,M) \to \THH(A^a,M^a)$.
\end{theorem}
\begin{example}\label{ex:Dx-THH-logification}
  The theorem implies that for the pre-log structure $(e,D(x))$ of
  Example~\ref{ex:Dx}, there is a stable equivalence $\THH(e,D(x)) \to
  \THH(e,j_*\!\GLoneJof(E))$, where we suppress a cofibrant replacement of the respective pre-log ring spectra from the notation. This equivalence is useful for
  calculations, since the homotopy type of $D(x)$ only depends on the
  degree of the periodicity element of $E$ and the homology of
  $D(x)_{h\cJ}$ is well understood.
\end{example}
The homotopy cartesian and cocartesian squares appearing in the next
proof and the subsequent lemma refer to the positive 
$\cJ$-model structure on $\cC\cS^{\cJ}$.

\begin{proof}[ 
of Theorem~\ref{thm:logification-inv-of-THH}]
  By Proposition~\ref{prop:replete-bar-as-repletion}, there is a chain
  of $\cJ$-equivalences relating $B^{\rep}(M)$ and
  $B^{\cy}(M)^{\rep}$. This chain of $\cJ$-equivalences can be
  chosen naturally with respect to the map $M \to M^a$ by forming the
  factorizations and lifts in Lemmas~\ref{lem:repletion-as-pullback}
  and~\ref{lem:BcyMgp-vs-BcyM-gp} in a model category of arrows. Using
  that the $\cJ$-equivalences in the chain are augmented over a
  cofibrant object, it follows from
  Corollary~\ref{cor:bSJ-preserves-we} that the map
  $\THH(A,M)\to\THH(A^a,M^{a})$ is stably equivalent to the map
 \[
    \THH(A)\!\sm_{\bS^{\cJ}[B^{\cy}\!(M)]}\bS^{\cJ}[B^{\cy}(M)^{\rep}] \to \THH(A^a)\!\sm_{\bS^{\cJ}[B^{\cy}(M^a)]}\bS^{\cJ}[B^{\cy}\!(M^a)^{\rep}] \, .
 \]
  Let $P$ be a cofibrant replacement of $\alpha^{-1}\GLoneJof(A)$ in
  the positive  $\cJ$-model structure, let $\rho \colon B^{\cy}
  (P) \to B^{\cy} (P)^{\rep}$ be the repletion map, and let $\gamma \colon
  B^{\cy} (P)\to B^{\cy} (P^{\gp})$ be the map induced by the group
  completion. We consider the pushout square
  \[\xymatrix@-1pc{
    B^{\cy} (P) \ar[d]_{\rho} \ar[r]^{\gamma} & \ar[d]^{\tau} B^{\cy} (P^{\gp})\\
    B^{\cy}(P)^{\rep} \ar[r] & B^{\cy}(P)^{\rep}\boxtimes_{B^{\cy} (P)}B^{\cy}
    (P^{\gp}).  }\] Since $\rho$ is an acyclic cofibration in the group
  completion model structure, so is $\tau$. The domain of $\tau$ is
  grouplike. Since $\rho$ is surjective on $\pi_0(-)_{h\cJ}$, the
  codomain of $\tau$ is also grouplike and it follows that $\tau$ is a
  $\cJ$-equivalence.

  Next we notice that the evident map $f\colon \bS^{\cJ}[B^{\cy} (P)] \to
  \THH(A)$ factors as a composite $\bS^{\cJ}[B^{\cy} (P)] \to
  \bS^{\cJ}[B^{\cy} (P^{\gp})] \to \THH(A)$ because $P \to
  \Omega^{\cJ}(A)$ factors through $ \GLoneJof(A)$ and therefore extends
  over $P \to P^{\gp}$. It follows that the cobase change of
  $\bS^{\cJ}[\rho]$ along $f$ is a stable equivalence, because it can
  be computed as the cobase change of the acyclic cofibration 
  $\bS^{\cJ}[\tau]$ along $\bS^{\cJ}[B^{\cy} (P^{\gp})] \to \THH(A) $.

  Let $\xymatrix@1{P \ar@{ >->}[r] & G \ar@{->>}[r]^-{\sim} &
    \GLoneJof(A) }$ be a factorization in the positive $\cJ$-model
  structure. Using the gluing lemma in $\cC\cS^{\cJ}$, the pushout of
  $M \ot P \to G$ is $\cJ$-equivalent to the logification $M^a$ from
  Construction~\ref{cons:logification}.  Applying the cyclic and
  repleted bar constructions and the functor $\bS^{\cJ}$ provides the
  following commutative cube in $\cC\Spsym$:
\[\xymatrix@-1.5pc{
  & \bS^{\cJ}[B^{\cy}(G)]\ar[rr] \ar'[d][dd]& & \bS^{\cJ}[B^{\cy}(M^a)]\ar[dd]\\
  \bS^{\cJ}[B^{\cy}(P)]  \ar[rr] \ar[ur] \ar[dd]_{\bS^{\cJ}[\rho]}& & \bS^{\cJ}[B^{\cy}(M)] \ar[ur] \ar[dd] & \\
  & \bS^{\cJ}[B^{\cy}(G)^{\rep}]  \ar'[r][rr]& & \bS^{\cJ}[B^{\cy}(M^a)^{\rep}] \, .\\
  \bS^{\cJ}[B^{\cy}(P)^{\rep}] \ar[rr] \ar[ur]& &
  \bS^{\cJ}[B^{\cy}(M)^{\rep}]\ar[ur] }\] The map
$\bS^{\cJ}[B^{\cy}(G)] \to \bS^{\cJ}[B^{\cy}(G)^{\rep}]$ is a stable
equivalence by Lemma~\ref{lem:Bcy--Brep-equiv-for-gplike}.
Combining this with the above statement about $\bS^{\cJ}[\rho]$, we
conclude that the left hand face becomes homotopy cocartesian after
cobase change along the canonical map $\bS^{\cJ}[B^{\cy}(G)] \to
\THH(A)$, in the sense that the cobase change of
\[
\bS^{\cJ}[B^{\cy}(G)]\sm_{\bS^{\cJ}[B^{\cy}(P)]} \bS^{\cJ}[B^{\cy}(P)^{\rep}] \to \bS^{\cJ}[B^{\cy}(G)^{\rep}]
\]
along $\bS^{\cJ}[B^{\cy}(G)] \to \THH(A)$ is a stable equivalence.  The top face is
homotopy cocartesian because $\bS^{\cJ}$ and $-\tensor S^1$ are left
Quillen. The bottom face is seen to be homotopy cocartesian by combining
Lemma~\ref{lem:repletion-hty-cocartesian} below with the fact that
$\bS^{\cJ}$ is left Quillen. By commuting homotopy pushouts, the statements about
the top and bottom face in the above cube imply that the diagram
\[
\xymatrix@-1pc{
\THH(A)\wedge_{\bS^{\cJ}[B^{\cy}(P)]}\bS^{\cJ}[B^{\cy}(P)^{\rep}]
\ar[r]\ar[d] &\THH(A)\wedge_{\bS^{\cJ}[B^{\cy}(M)]}\bS^{\cJ}[B^{\cy}(M)^{\rep}] \ar[d]\\
\THH(A)\wedge_{\bS^{\cJ}[B^{\cy}(G)]}\bS^{\cJ}[B^{\cy}(G)^{\rep}]
\ar[r] & \THH(A^a)\wedge_{\bS^{\cJ}[B^{\cy}(M^a)]}\bS^{\cJ}[B^{\cy}(M^a)^{\rep}]
}
\]
induced by the logification construction is homotopy cocartesian. This gives the result since the vertical map on the left is a stable equivalence.
\qed \end{proof}
The next lemma was used in the previous proof.
\begin{lemma}\label{lem:repletion-hty-cocartesian}
  Consider the following commutative diagram of cofibrant objects in $\cC\cS^{\cJ}$ in
  which the $N_i^{\rep}$ are the repletions of the horizontal
  composites $\epsilon_i \colon N_i \to M_i$:
  \[\xymatrix@-1.5pc{
    & N_2 \ar'[d][dd] \ar[rr]& & N_2^{\rep} \ar'[d][dd] \ar[rr]& & M_2\ar[dd]\\
    N_1 \ar@{ >->}[ur] \ar[dd] \ar[rr]& & N_1^{\rep} \ar[ur] \ar[dd] \ar[rr]& & M_1\ar@{ >->}[ur] \ar[dd]& \\
    & N_4 \ar'[r][rr] & & N_4^{\rep} \ar'[r][rr] & & M_4.\\
    N_3 \ar@{ >->}[ur] \ar[rr]& & N_3^{\rep}\ar[ur] \ar[rr] & & M_3 \ar@{ >->}[ur]
  }\] If the left and right hand faces are homotopy cocartesian squares and
  all $\epsilon_i$ are virtually surjective, then the middle square of
  repletions is homotopy cocartesian.
\end{lemma}
\begin{proof}
  We use the characterization of the repletion from
  Corollary~\ref{cor:repletion-as-hty-pb}.  Group completion preserves
  homotopy cocartesian squares because the homotopy pushout of
  grouplike objects in $\cC\cS^{\cJ}$ is grouplike. Since the
  two-sided bar construction over $\boxtimes$ can be used to compute
  homotopy pushouts, this reduces the statement to showing that
\[\xymatrix@-1pc{
  B^{\boxtimes}(N_3^{\rep},N_1^{\rep},N_2^{\rep}) \ar[r] \ar[d] & B^{\boxtimes}(N_3^{\gp},N_1^{\gp},N_2^{\gp}) \ar[d]\\
  B^{\boxtimes}(M_3,M_1,M_2) \ar[r] &
  B^{\boxtimes}(M_3^{\gp},M_1^{\gp},M_2^{\gp}) }\] is homotopy
cartesian. By~\cite[Corollary 11.4]{Sagave-S_diagram}, a square of
$\cJ$-spaces is homotopy cartesian if and only if it becomes  a homotopy
cartesian diagram of simplicial sets after applying $(-)_{h\cJ}$.  Using that
the monoidal structure map of $(-)_{h\cJ}$ is a $\cJ$-equivalence for
cofibrant objects (see e.g.~\cite[Lemma
2.11]{Sagave_log-on-k-theory}), it remains to show that the following
square of simplicial sets is homotopy cartesian:
\[\xymatrix@-1pc{
  B((N_3^{\rep})_{h\cJ},(N_1^{\rep})_{h\cJ},(N_2^{\rep})_{h\cJ})
  \ar[r] \ar[d] &
  B((N_3^{\gp})_{h\cJ},(N_1^{\gp})_{h\cJ},(N_2^{\gp})_{h\cJ}) \ar[d]\\
  B((M_3)_{h\cJ},(M_1)_{h\cJ},(M_2)_{h\cJ}) \ar[r] &
  B((M_3^{\gp})_{h\cJ},(M_1^{\gp})_{h\cJ},(M_2^{\gp})_{h\cJ}).
}\] Using the hypothesis of virtual surjectivity, this follows from
the Bousfield--Fried\-lander Theorem~\cite[Theorem
B.4]{Bousfield-F_Gamma-bisimplicial} about the realization of
pointwise homotopy cartesian squares of bisimplicial sets: The square
of bisimplicial sets obtained by applying $B_{\bullet}$ in
the situation above is pointwise homotopy cartesian,
\[
B_{\bullet}((N_3^{\gp})_{h\cJ},(N_1^{\gp})_{h\cJ},(N_2^{\gp})_{h\cJ})
\quad\text{and}\quad
B_{\bullet}((M_3^{\gp})_{h\cJ},(M_1^{\gp})_{h\cJ},(M_2^{\gp})_{h\cJ})\]
satisfy the $\pi_*$-Kan condition because all simplicial monoids
involved are grouplike, and the assumed virtual surjectivity implies that
the map between these two bisimplicial sets induces a Kan fibration on
vertical path components.
\qed \end{proof}

\section{The case of discrete pre-log rings}\label{sec:discrete-pre-log-rings}
A \emph{discrete pre-log ring} $(B,N,\beta)$ is a commutative ring $B$
together with a commutative monoid $N$ and a monoid homomorphism
$\beta \: N \to (B, \cdot)$. These objects are usually just called
\emph{pre-log rings}, but we use the additional term \emph{discrete}
to distinguish them from pre-log ring spectra. In this section we
explain how discrete pre-log rings can be viewed as
pre-log ring spectra, construct a homotopy cofiber sequence for the
log $\THH$ of $(B,N,\beta)$ if $N$ is a free commutative monoid
$\langle x\rangle$ on a single generator $x$, and compute the mod $p$
homotopy of $\THH(\bZ,\langle p \rangle)$. This is interesting in its
own right and serves as an example of the general principle in
Proposition~\ref{prop:general-logTHH-homotopy-cofiber} for setting up
homotopy cofiber sequences.

Let $(B,N,\beta)$ be a discrete pre-log ring. The standard model for
the Eilenberg--Mac\,Lane spectrum associated with $B$ is the
commutative symmetric ring spectrum $HB$ with $(HB)_n = \widetilde
B[S^n]$, the reduced $B$-linearization of the $n$-sphere. Since the
underlying multiplicative monoid of $0$-simplices of $(HB)_0$ is the
multiplicative monoid of $B$, the structure map $\beta$ provides a
morphism of discrete simplicial monoids $N \to (HB)_0 =
(\Omega^{\cJ}(HB))(\bld{0},\bld{0})$.  In the following we write
$\Fzero = F^{\cJ}_{(\bld{0},\bld{0})}$ for the strong symmetric
monoidal functor $\cS \to \cS^{\cJ}$ that is left adjoint to the
evaluation functor at $(\bld{0},\bld{0})$. With this notation, we obtain
a map $\beta\colon \Fzero{}N\to \Omega^{\cJ}(HB)$ in $\cC\cS^{\cJ}$
that defines the pre-log ring spectrum $(HB, \Fzero{}N)$ associated
with $(B,N)$. Here the monoid structure of $FN$ is provided by the
paring
\[ \Fzero{}N\boxtimes \Fzero{}N = F^\cJ_{(\bzero,\bzero)}(N) \boxtimes F^\cJ_{(\bzero,\bzero)}(N)
\cong F^\cJ_{(\bzero,\bzero)}(N \times N) \to
F^\cJ_{(\bzero,\bzero)}(N)=\Fzero{}N,\] and $\bS^{\cJ}[FN] \cong \bS[N]$, where $\bS[-] =
\Sigma^\infty(-)_+$ denotes the unreduced symmetric suspension
spectrum of an ordinary simplicial set.

\subsection{The cyclic and the replete bar construction for discrete monoids}\label{subsec:cy-rep-discrete}
If~$N$ is a commutative monoid, we can view it as a discrete
simplicial commutative monoid and form the cyclic bar construction
$B^{\cy}(N)$ in simplicial commutative monoids. Using the algebraic
group completion $N \to N^{\gp}$ of $N$, we define the replete bar
construction $B^{\rep}(N)$ as the pullback of $N \to N^{\gp} \ot
B^{\cy}(N^{\gp})$ in simplicial commutative monoids. To relate this to
the replete bar construction in commutative $\cJ$-space monoids, we
first note that the strong symmetric monoidal functor $F$ induces a natural isomorphism $F(B^{\cy}(N))\cong B^{\cy}(FN)$. Secondly, the
map $\Fzero{}N\to \Fzero{}(N^{\gp})$ has grouplike
codomain and induces a weak equivalence when applying
$B((-)_{h\cJ})$ by \cite[Proposition Q.1]{Friedlander-M_filtrations}. So after
composing with a fibrant replacement $(-)^{\mathrm{fib}}$ in the positive $\cJ$-model structure, it provides a group completion $\Fzero{}N \to
(\Fzero{}(N^{\gp}))^{\mathrm{fib}}$ of commutative $\cJ$-space
monoids. We also note that even though the underlying $\cJ$-space is cofibrant, $FN$ is usually not cofibrant in the positive  $\cJ$-model structure on $\cC\cS^{\cJ}$. Thus, for the purpose of defining log $\THH$, we should first pass to a cofibrant replacement $(FN)^{\mathrm{cof}}$.
\begin{lemma}\label{lem:cof-rep-of-discrete-on-Bcy-Brep}
  Let $N$ be a commutative monoid, and let $(\Fzero{}N)^{\mathrm{cof}}
  \to \Fzero{}N$ be a cofibrant replacement in the positive 
  $\cJ$-model structure.
  \begin{enumerate}[(i)]
  \item The induced map $B^{\cy}((FN)^{\mathrm{cof}}) \to
    B^{\cy}(FN)$ is a $\cJ$-equivalence and gives rise to a stable equivalence
\[
\bS^{\cJ}[ B^{\cy}((FN)^{\mathrm{cof}})] \xrightarrow{\simeq}
\bS^{\cJ}[B^{\cy}(FN)]\cong \bS[B^{\cy}(N)]
\] 
of  commutative $\bS^{\cJ}[ B^{\cy}((FN)^{\mathrm{cof}})]$-algebras with circle action.  
  \item There is a chain of $\cJ$-equivalences of commutative
  $B^{\cy}((\Fzero{}N)^{\mathrm{cof}})$-algebras, connecting
  $B^{\rep}((\Fzero{}N)^{\mathrm{cof}})$ to $\Fzero{}(B^{\rep}(N))$,
  that induces a chain of stable equivalences
\[
\bS^{\cJ}[B^{\rep}((\Fzero{}N)^{\mathrm{cof}})]\simeq 
\bS^{\cJ}[\Fzero{}(B^{\rep}(N))]\cong \bS[B^{\rep}(N)] 
\]    
of commutative $\bS^{\cJ}[ B^{\cy}((FN)^{\mathrm{cof}})]$-algebras with circle action. 
  \end{enumerate}
\end{lemma}
\begin{proof}
  The commutative $\cJ$-space monoid $\Fzero{}N$ is flat by the definition of the generating cofibrations for the flat model structure on $\cJ$-spaces~\cite[Corollary 5.10 and Proposition 6.16]{Sagave-S_diagram}. Hence it follows from Lemma~\ref{lem:properties-of-flat} that the $\cJ$-equivalence $(\Fzero{}N)^{\mathrm{cof}}\to \Fzero{}N$ induces a $\cJ$-equivalence of cyclic bar constructions. Since $B^{\cy}(FN)$ is augmented over the monoidal unit $U^{\cJ}$, it follows from Corollary~\ref{cor:bSJ-preserves-we} that $\bS^{\cJ}$ takes the latter $\cJ$-equivalence to a stable equivalence. This implies (i).
 
  For part (ii) we observe that there is a commutative diagram 
  \[
  \xymatrix@-1pc{
      \Fzero{}N \ar[d] & (\Fzero{}N)^{\mathrm{cof}} \ar[l] \ar[r]^{=} \ar[d] & (\Fzero{}N)^{\mathrm{cof}} \ar[d] & (\Fzero{}N)^{\mathrm{cof}} \ar[d] \ar[l]_{=} \\
      \Fzero{}(N^{\gp})  & (\Fzero{}(N^{\gp}))^{\mathrm{cof}} \ar[l] \ar[r]&  (\Fzero{}(N^{\gp})^{\mathrm{cof}})^{\mathrm{fib}}  & ((\Fzero{}N)^{\mathrm{cof}})^{\gp}. \ar[l]
    }  
  \]
  Here the left hand square and the middle square are induced by the
  cofibrant and fibrant replacements in the positive $\cJ$-model
  structure, respectively, and the lower horizontal map in the right
  hand square results from the lifting property in the group
  completion model structure by the remark before the lemma. Each
  vertical map in the diagram induces a weak equivalence when applying
  $B((-)_{h\cJ})$, and all terms in the lower row are grouplike.
  Since $(\Fzero{}N)^{\mathrm{cof}} \to \Fzero{}N$ is a
  $\cJ$-equivalence, it follows that all horizontal maps in the
  diagram are $\cJ$-equivalences.

  Next we apply the cyclic bar construction to the terms in the lower
  row of the diagram and form the homotopy pullbacks of the vertical
  maps and the respective augmentations of the cyclic bar
  constructions to obtain a chain of $\cJ$-equivalences between the
  homotopy pullbacks. Since the map of discrete simplicial sets $N \to
  N^{\gp}$ is a Kan fibration, the pullback defining
  $B^{\rep}(N)$ is already a homotopy pullback. By~\cite[Corollary
  11.4]{Sagave-S_diagram}, the functor $\Fzero{}$ maps it to a
  homotopy cartesian square in $\cC\cS^{\cJ}$. So we have constructed
  the desired chain of $\cJ$-equivalences, and arguing as in part (i), we see that $\bS^{\cJ}$ takes this to a chain of stable equivalences.
    \qed \end{proof}

Let $(A,M,\alpha)$ be a pre-log ring spectrum with $M = \Fzero{}N$ for
a discrete commutative monoid $N$.  Choosing a factorization
of the adjoint structure map $\bS[N] \to A$ into a cofibration $\bar \alpha^c\colon \bS[N]
\to A^{c}$ followed by an acyclic fibration $A^{c} \to A$ in
commutative symmetric ring spectra, we can give the following simpler
description of $\THH(A,M)$ that does not involve a cofibrant replacement
of~$(A,M)$. This shows in particular that for a discrete pre-log ring $(B,N)$, the definition of log 
$\THH$ via the pre-log ring spectrum $(HB,FN)$ is equivalent to the definition given by the first author in \cite{Rognes_TLS}. 

\begin{proposition}\label{prop:discrete-logTHH-description}
Let $(A,M)$ be a pre-log ring spectrum with $M=FN$ for a discrete commutative monoid $N$, and let $(A^{\mathrm{cof}},M^{\mathrm{cof}})$ be a cofibrant replacement. Then there is a chain of stable equivalences 
\[
\THH(A^{\mathrm{cof}},M^{\mathrm{cof}})\simeq
 B^{\cy}(A^c)\sm_{\bS[B^{\cy}(N)]}\bS[B^{\rep}(N)]
  \]
of $\THH(A^{\mathrm{cof}})$-modules with circle action.
\end{proposition}
\begin{proof}
We first use the lifting axiom for the positive stable model structure on $\cC\Spsym$ to get a stable equivalence $A^{\mathrm{cof}}\to A^c$ of commutative $\bS^{\cJ}[M^{\mathrm{cof}}]$-algebras, which in turn induces a stable equivalence 
$B^{\cy}(A^{\mathrm{cof}})\to B^{\cy}(A^c)$ of commutative
$\bS^{\cJ}[B^{\cy}(M^{\mathrm{cof}})]$-algebras. It now follows from left properness of the positive stable model structure that the stable equivalences in Lemma~\ref{lem:cof-rep-of-discrete-on-Bcy-Brep} give rise to a chain of stable equivalences as stated in the proposition. 
\qed \end{proof}

\subsection{Free monoids and log \texorpdfstring{$\THH$}{THH} localization sequences}\label{subsec:discrete-localization}
In order to prepare for subsequent computations of log $\THH$, we
describe the homotopy type and the homology of $B^{\cy}(N)$ and
$B^{\rep}(N)$ (or rather of their geometric realizations) for a free commutative monoid $N$ on one generator. We write $H_* X = H_*(X; \bF_p)$ for (mod~$p$) homology, and
denote polynomial, height $h$ truncated polynomial, exterior and
divided power algebras by $P(x)$, $P_h(x)$, $E(x)$ and
$\Gamma(x)=\bF_p\{\gamma_k x \, | \, k\geq 0 \}$, respectively.

Let $N = \langle x\rangle = \{x^k \mid k\ge0\}$ be the free
commutative monoid on the generator~$x$.  Its cyclic bar construction
decomposes as a disjoint union
\begin{equation}\label{eq:Bcy-of-free-com-mon}
B^{\cy}(N) = \coprod_{k\ge0} B^{\cy}(N; k)
	\simeq {*} \sqcup \coprod_{k\ge1} S^1(k) \,,
\end{equation}
where $B^{\cy}(N; k)$ denotes the geometric realization of the cyclic subset consisting of the
simplices $(x^{i_0}, \dots, x^{i_q})$ with $i_0 + \dots + i_q = k$.
Here $B^{\cy}(N; 0)$ is a point, while $B^{\cy}(N; k)$ for $k\ge1$
is equivariantly homotopy equivalent to $S^1$ with the degree~$k$ circle action, which we write as $S^1(k)$. Hence
\begin{equation}\label{eq:homology-of-Bcy-of-free-com-mon}
H_* B^{\cy}(N) \cong P(x) \otimes E(dx) \,,
\end{equation}
with $x \in H_0 S^1(1)$ and $dx \in H_1 S^1(1)$ represented by the
cycles $(x)$ and $(1, x)$, respectively.
See~\cite[Proposition~3.20]{Rognes_TLS}
or \cite[Lemma 2.2.3]{Hesselholt_p-typical}
for further details.

The group completion $N^{\gp} = \langle x^{\pm1}\rangle = \{x^k
\mid k\in\bZ\}$ has cyclic bar construction $B^{\cy}(N^{\gp}) =
\coprod_{k\in\bZ} B^{\cy}(N^{\gp}; k)$, containing the replete bar
construction
\begin{equation}\label{eq:Brep-of-free-com-mon}
  B^{\rep}(N) = \coprod_{k\ge0} B^{\cy}(N^{\gp}; k)
  \simeq \coprod_{k\ge0} S^1(k)
\end{equation}
as the non-negatively indexed summands.  The repletion map restricts
to equivalences $B^{\cy}(N; k) \simeq B^{\cy}(N^{\gp}; k)$ for
$k\ge1$, while $B^{\cy}(N^{\gp}; 0) \cong B(N^{\gp})$ is related to $S^1(0)$ (that is, $S^1$ equipped with the trivial circle action) by a chain of weak equivalences that are circle equivariant. Hence
\[
H_* B^{\rep}(N) \cong P(x) \otimes E(d\log x) \,,
\]
with $x \in H_0 S^1(1)$ and $d\log x \in H_1 S^1(0)$ represented
by the cycles $(x)$ and $(x^{-1}, x)$, respectively.
In homology, $\rho_*(x) = x$ and $\rho_*(dx)= x \cdot d\log x$.
All this can be deduced from \cite[Proposition~3.21]{Rognes_TLS}.
The next lemma is an immediate consequence of the equivalences 
in \eqref{eq:Bcy-of-free-com-mon} and \eqref{eq:Brep-of-free-com-mon}.
\begin{lemma} \label{lem:bcyncof}
For $N = \langle x\rangle$ there is a homotopy cofiber sequence
\[
\bS[B^{\cy}(N)]
	\overset{\rho}\longto \bS[B^{\rep}(N)]
 \overset{\partial}\longto        \Sigma \bS
\]
of $\bS[B^{\cy}(N)]$-modules with circle action, in which 
the action on $\Sigma \bS$ is trivial. \qed
\end{lemma}

Notice that the commutative $\cJ$-space monoid $F\langle x\rangle$, obtained by applying the strong symmetric monoidal functor $F$ to $\langle x\rangle$, can be identified with the free commutative $\cJ$-space monoid $\bC\langle x\rangle$ on a generator in bidegree $(\bld 0,\bld 0)$, cf.\ Example~\ref{ex:free-J-space}. Given a commutative symmetric ring spectrum $A$, every $0$-simplex $x\in A_0$ thus gives rise to a pre-log ring spectrum $(A,\bC\langle x\rangle)$. This will usually not be cofibrant, so in order to define log $\THH$ we should first pass to a cofibrant replacement $(A^{\mathrm{cof}},\bC\langle x\rangle^{\mathrm{cof}})$. In order to form the balanced smash product $\modcofelement{A}{x}$, 
we view $\bS$ as a commutative $\bS^{\cJ}[\bC\langle x\rangle^{\mathrm{cof}}]$-algebra via the composition $\bC\langle x\rangle^{\mathrm{cof}}\to \bC\langle x\rangle\to U^{\cJ}$, using the isomorphism $\bS^{\cJ}[U^{\cJ}]\iso \bS$. The next lemma shows that we may view $\modcofelement{A}{x}$ as a homotopy invariant model for the quotient of $A$ by the ideal generated by $x$.
\begin{lemma}
For a cofibrant replacement $(A^{\mathrm{cof}},\bC\langle x\rangle^{\mathrm{cof}})\to (A,\bC\langle x\rangle)$
there is a homotopy cofiber sequence 
$
A\xrightarrow{x} A\to \modcofelement{A}{x} 
$
of $A^{\mathrm{cof}}$-modules.
\end{lemma}
\begin{proof}
The conditions for $(A^{\mathrm{cof}},\bC\langle x\rangle^{\mathrm{cof}})$ to be a cofibrant replacement imply that the induced map 
$A^{\mathrm{cof}}\wedge_{\bS^{\cJ}[\bC\langle x\rangle^{\mathrm{cof}}]}
\bS[\langle x\rangle]\to A$ is a stable equivalence. Using this observation and Lemma~\ref{lem:ext-along-cof-pre-st-eq}, the homotopy cofiber sequence in the lemma is obtained by applying the functor  
$A^{\mathrm{cof}}\wedge_{\bS^{\cJ}[\bC\langle x\rangle^{\mathrm{cof}}]}(-)$ to the obvious homotopy cofiber sequence $\bS[\langle x\rangle]\xrightarrow{x}
\bS[\langle x\rangle]\to \bS$ of $\bS[\langle x\rangle]$-modules.
\qed \end{proof}

Specializing the general principle in Proposition~\ref{prop:general-logTHH-homotopy-cofiber} to the case of the commutative $\cJ$-space monoid $\bC\langle x\rangle$, we get the localization sequence in the next theorem.

\begin{theorem}\label{thm:discrete-log-THH-localization-sequence}
Let $A$ be a commutative symmetric ring spectrum, and let $(A,\bC\langle x\rangle)$ be the pre-log ring spectrum determined by a $0$-simplex $x\in A_0$. Then there is a homotopy cofiber sequence
\[
\THH(A^{\mathrm{cof}})\overset{\rho}{\longto}\THH(A^{\mathrm{cof}},\bC\langle x\rangle^{\mathrm{cof}})\overset{\partial}{\longto}\Sigma\THH(\modcofelement{A}{x})
\]
of $\THH(A^{\mathrm{cof}})$-modules with circle action.
\end{theorem}
\begin{proof}
With the notation from Proposition~\ref{prop:general-logTHH-homotopy-cofiber}, let $P=U^{\cJ}$, and consider the map $\bC\langle x\rangle^{\mathrm{cof}}\to 
\bC\langle x\rangle \to U^{\cJ}$. Using the stable equivalences in 
Lemma~\ref{lem:cof-rep-of-discrete-on-Bcy-Brep}, the  
homotopy cofiber sequence from Lemma~\ref{lem:bcyncof} translates into a homotopy cofiber sequence 
\[
\bS^{\cJ}[B^{\cy}(\bC\langle x\rangle^{\mathrm{cof}})]\overset{\rho}{\longto} 
\bS^{\cJ}[B^{\rep}(\bC\langle x\rangle^{\mathrm{cof}})]\overset{\partial}{\longto} 
\Sigma\bS
\]
of $\bS^{\cJ}[B^{\cy}(\bC\langle x\rangle^{\mathrm{cof}})]$-modules with circle action. Because of the canonical isomorphism $\bS \cong \bS^{\cJ}[B^{\cy}(U^{\cJ})]$, this provides the necessary input for Proposition~\ref{prop:general-logTHH-homotopy-cofiber}.
\qed \end{proof}

\begin{remark}
Using the description in Proposition~\ref{prop:discrete-logTHH-description}, one can argue as in the above theorem to get a homotopy cofiber sequence
\[
B^{\cy}(A^c)\overset{\rho}{\longto} B^{\cy}(A^c)\wedge_{\bS[B^{\cy}(\langle x\rangle)]}\bS[B^{\rep}(\langle x\rangle)]\overset{\partial}{\longto} \Sigma
B^{\cy}(A^c\wedge_{\bS[\langle x\rangle]}\bS),
\]
which is stably equivalent to the one in the theorem. 
\end{remark}

If $(B,N)$ is a discrete pre-log ring, we shall simplify the notation
by writing $\THH(B,N)$ for
$\THH((HB)^{\mathrm{cof}},(FN)^{\mathrm{cof}})$ and $\THH(B)$ for
$\THH((HB)^{\mathrm{cof}})$. We remark that since the underlying
symmetric spectrum of $HB$ is flat ~\cite[Proposition~I.7.14(ii) and
Example~I.7.33]{Schwede_SymSp}, we could equally well use $B^{\cy}(HB)$
as a model for~$\THH(B)$.

\begin{example}\label{ex:EM-localization}
Let $B$ be a commutative ring and $x \in B$ an element that does not
divide zero.  Then $\modcofelement{(HB)}{x}\simeq H(B/(x))$, and there is a homotopy
cofiber sequence
\[
\THH(B) \overset{\rho}\longto \THH(B, \langle x\rangle)
	\longto \Sigma \THH(B/(x))
\]
of $\THH(B)$-modules with circle action.  In particular, when $B$ is a
discrete valuation ring with uniformizer $\pi$, residue field $\ell =
B/(\pi)$ and fraction field $L = B[\pi^{-1}]$, there is a homotopy
cofiber sequence
\[
\THH(B) \overset{\rho}\longto \THH(B, \langle\pi\rangle)
	\longto \Sigma \THH(\ell)
\]
of $\THH(B)$-modules with circle action. We expect that it agrees with
the homotopy cofiber sequence
\[
\THH(B) \overset{j^*}\longto \THH(B|L)
	\overset{\partial}\longto \Sigma \THH(\ell)
\]
of Hesselholt and Madsen \cite[\S1.5]{Hesselholt-M_local_fields},
which is only defined in this more restricted setting. 
\end{example}

\begin{remark}
  A commutative symmetric ring spectrum $A$ that does not have the
  homotopy type of an Eilenberg--Mac\,Lane spectrum will often not
  admit an interesting pre-log structure with $M = \bC\langle
  x\rangle$. One reason is that for this to exist, the higher
  multiplicative Dyer--Lashof operations on the image of $x$ in
  $H_0(\Omega^{\infty} A;\bF_p)$ must vanish. For example, this is not
  the case for the element $[p]$ in $H_0(\Omega^{\infty}ku;\bF_p)$;
  see~\cite[Lemma~9.6]{Rognes_TLS}
and the discussion in \cite[Remark 9.17]{Rognes_TLS}.
\end{remark}
\begin{example}
  In the case of the canonical pre-log ring spectrum structure on $A =
  \bS^{\cJ}[\bC\langle x\rangle] \cong \bS[\langle x\rangle]$, with $M
  = \bC\langle x\rangle$, we have stable equivalences
\[
\THH(\bS[\langle x\rangle]^{\mathrm{cof}})
\simeq \bS[B^{\cy} \langle x\rangle]
\simeq \bS[{*} \sqcup \coprod_{k\ge1} S^1(k)]
\]
and
\[
\THH(\bS[\langle x\rangle]^{\mathrm{cof}}, \bC\langle x\rangle^{\mathrm{cof}})
\simeq \bS[B^{\rep} \langle x\rangle]
\simeq \bS[\coprod_{k\ge0} S^1(k)]
\simeq \bS[\Lambda_{\ge0} S^1] \,,
\]
of commutative symmetric ring spectra with circle action, by
Proposition~\ref{prop:discrete-logTHH-description} and
equations~\eqref{eq:Bcy-of-free-com-mon}
and~\eqref{eq:Brep-of-free-com-mon}.  Here $\Lambda_{\ge0} S^1$ is the
subspace of the free loop space $\Lambda S^1$ that consists of maps
$S^1 \to S^1$ of non-negative degree.  By the discussion before
Lemma~\ref{lem:cof-rep-of-discrete-on-Bcy-Brep}, the commutative
$\cJ$-space monoid $F\langle x^{\pm1}\rangle$ is $\cJ$-equivalent to $\C{x}^{\gp}$.
Hence $\bS^{\cJ}[\bC\langle x\rangle^{\gp}]$ is stably equivalent to $\bS[\langle
x^{\pm1}\rangle]$, and the natural map to the localization $A[M^{-1}] =
\bS^{\cJ}[\bC\langle x\rangle^{\gp}]\simeq \bS[\langle
x^{\pm1}\rangle]$ induces the evident inclusions to
\[
\THH(\bS[\langle x^{\pm1}\rangle]^{\mathrm{cof}})
\simeq \bS[B^{\cy} \langle x^{\pm1}\rangle]
\simeq \bS[\coprod_{k\in\bZ} S^1(k)]
\simeq \bS[\Lambda S^1] \,.
\]
\end{example}

\subsection{Logarithmic \texorpdfstring{$\THH$}{THH} of \texorpdfstring{$(\bZ,\langle p\rangle)$}{(Z,<p>)}}
Let us write $\bar\pi_* X = \pi_*(X; \bZ/p)$ for mod~$p$ homotopy,
where $p$ is a prime, and let~$\doteq$ denote equality up to a
unit in $\bF_p$. In this section all balanced smash products of symmetric spectra should be understood in the homotopy invariant left derived sense. 
We begin by determining the structure of a K{\"u}nneth
spectral sequence obtained by identifying $\THH(\bZ/p)$ with the left derived balanced smash 
product $\THH(\bZ)\wedge_{\bS[B^{\cy}\langle p\rangle]} \bS$, where we use $B^{\cy}(H\bZ)$ as a model for $\THH(\bZ)$, cf.\ the remarks preceding Example~\ref{ex:EM-localization}.

\begin{proposition}
Consider the case $B = \bZ$ and $N = \langle p\rangle$.  There is an
algebra spectral sequence
\begin{align*}
E^2_{**} &= \Tor^{H_* B^{\cy}\langle p\rangle}_{**}(\bar\pi_* \THH(\bZ), \bF_p) \\
	& \Longrightarrow \bar\pi_* \THH(\bZ/p) \,,
\end{align*}
where $H_* B^{\cy}\langle p\rangle = P(p) \otimes E(dp)$
by~\eqref{eq:homology-of-Bcy-of-free-com-mon}, $\bar\pi_* \THH(\bZ) =
E(\lambda_1) \otimes P(\mu_1)$ and $\bar\pi_* \THH(\bZ/p) =
E(\epsilon_0) \otimes P(\mu_0)$, with $|p| = 0$, $|dp| = 1$,
$|\lambda_1| = 2p-1$, $|\mu_1| = 2p$, $|\epsilon_0| = 1$ and $|\mu_0|
= 2$.  Here
\[
E^2_{**} = E(\lambda_1) \otimes P(\mu_1) \otimes E([p])
	\otimes \Gamma([dp]) \,,
\]
where $[p]$ has bidegree $(1,0)$ and $[dp]$ has bidegree $(1,1)$.
There are non-trivial $d^p$-differentials
\[
d^p(\gamma_k [dp]) \doteq \lambda_1 \cdot \gamma_{k-p} [dp]
\]
for all $k\ge p$, leaving
\[
E^\infty_{**} = P(\mu_1) \otimes E([p]) \otimes P_p([dp]) \,.
\]
Hence $[p]$ represents $\epsilon_0$, $[dp]$ represents $\mu_0$, and $\mu_1$
represents $\mu_0^p$ (up to units in $\bF_p$) in the abutment, and
there is a multiplicative extension $[dp]^p \doteq \mu_1$.
\end{proposition}

\begin{proof}
Applying cobase change along $\bS \to H\bZ \to H = H\bF_p$ to the left derived balanced smash product mentioned above, we get a homotopy cocartesian square
\begin{equation} \label{xy:zmodp}
\xymatrix@-1pc{
H \wedge B^{\cy}\langle p\rangle_+ \ar[rr] \ar[d] && H \ar[d] \\
H \wedge_{H\bZ} \THH(\bZ) \ar[rr] && H \wedge_{H\bZ} \THH(\bZ/p)
}
\end{equation}
of commutative $H$-algebras.  The spectral sequence in question is the
associated $\Tor$ spectral sequence \cite[Theorem~IV.4.1]{EKMM}, in view of
the identification $\pi_*(H \wedge_{H\bZ} X) \cong \bar\pi_* X$
for $H\bZ$-modules.

B{\"o}kstedt computed that $\pi_* \THH(\bZ/p) = P(\mu_0)$ with
$|\mu_0|=2$, so $\bar\pi_* \THH(\bZ/p) = E(\epsilon_0) \otimes
P(\mu_0)$ with $|\epsilon_0| = 1$ (the mod~$p$ Bockstein element).
B{\"o}kstedt also computed that $\bar\pi_* \THH(\bZ) = E(\lambda_1)
\otimes P(\mu_1)$.  For published proofs, see e.g.~the cases $m=0$ and
$m=1$ of \cite[Theorem~5.12]{Angeltveit-R_Hopf-algebra}, or \cite[\S3, \S4]{Ausoni-R_K-Morava}.
This leads to the stated $E^2$-term and abutment.

The algebra generators $\lambda_1$, $\mu_1$, $[p]$ and $[dp]$ must be
infinite cycles for filtration reasons.  To determine the differentials on
the remaining algebra generators, i.e., the divided powers $\gamma_{p^i}[dp]$
for $i\ge1$, we note that the abutment $E(\epsilon_0) \otimes P(\mu_0)$
has exactly one generator in each non-negative degree.  In total
degree~$2p-1$ the $E^2$-term is generated by $\lambda_1$ and $[p] \cdot
\gamma_{p-1}[dp]$.  Hence one of these must be hit by a differential,
and for filtration reasons the only possibility is $d^p(\gamma_p[dp])
\doteq \lambda_1$.

Taking this into account, it follows that $[p]$, $[dp]$ and $\mu_1$
survive to $E^\infty$, where they must represent $\epsilon_0$, $\mu_0$
and $\mu_0^p$, respectively, up to  units in $\bF_p$.
Furthermore, in degree~$2p^2-1$ the only remaining generators are
\[
\mu_1^{p-1} \cdot [p] \cdot \gamma_{p-1}[dp]
\qquad\text{and}\qquad
\lambda_1 \cdot \gamma_{p^2-p}[dp] \,.
\]
The first of these must represent a unit in $\bF_p$ times
$(\mu_0^p)^{p-1} \cdot \epsilon_0 \cdot \mu_0^{p-1}$ in the abutment.
Since this product is nonzero, the first generator cannot be a boundary,
and must survive to $E^\infty$.
Hence the second generator must be hit by a differential, which must
come from $\gamma_{p^2}[dp]$.  This explains the second differential,
$d^p(\gamma_{p^2}[dp]) \doteq \lambda_1 \cdot \gamma_{p^2-p}[dp]$.
The cases $i\ge3$ are very similar.
\qed \end{proof}

Having determined the differentials and multiplicative extensions
above, we can now analyze the K{\"u}nneth spectral sequence
computing $\bar\pi_* \THH(\bZ, \langle p\rangle)$.

\begin{proposition}
Consider the case $B = \bZ$ and $N = \langle p\rangle$.  There is an
algebra spectral sequence
\begin{align*}
E^2_{**} &= \Tor^{H_* B^{\cy}\langle p\rangle}_{**}
	(\bar\pi_* \THH(\bZ), H_* B^{\rep}\langle p\rangle) \\
	& \Longrightarrow \bar\pi_* \THH(\bZ, \langle p\rangle) \,,
\end{align*}
where $H_* B^{\cy}\langle p\rangle = P(p) \otimes E(dp)$
by~\eqref{eq:homology-of-Bcy-of-free-com-mon}, $\bar\pi_* \THH(\bZ) =
E(\lambda_1) \otimes P(\mu_1)$ and $H_* B^{\rep}\langle p\rangle = P(p) \otimes
E(d\log p)$, with $|d\log p| = 1$ and the remaining degrees as above.
Here
\[
E^2_{**} = E(\lambda_1) \otimes P(\mu_1)
	\otimes E(d\log p) \otimes \Gamma([dp])
\]
where $[dp]$ has bidegree~$(1,1)$.  There are non-trivial differentials
\[
d^p(\gamma_k [dp]) \doteq \lambda_1 \cdot \gamma_{k-p} [dp]
\]
for all $k\ge p$, leaving
\[
E^\infty_{**} = P(\mu_1) \otimes E(d\log p) \otimes P_p([dp]) \,.
\]
There is a multiplicative extension $[dp]^p \doteq \mu_1$, so
the abutment is
\[
\bar\pi_* \THH(\bZ, \langle p\rangle) = E(d\log p) \otimes P(\kappa_0)
\]
where $\kappa_0$ is represented by $[dp]$, with $|\kappa_0| = 2$.
\end{proposition}

\begin{proof}
Applying cobase change along $\bS \to H\bZ \to H = H\bF_p$ to the
balanced smash product defining $\THH(\bZ, \langle p\rangle)$, we get a homotopy
cocartesian square
\begin{equation}\label{xy:zn}\begin{split} 
\xymatrix@-1pc{
H \wedge B^{\cy}\langle p\rangle_+ \ar[rr]^-{\rho} \ar[d]
	&& H \wedge B^{\rep}\langle p\rangle_+ \ar[d] \\
H \wedge_{H\bZ} \THH(\bZ) \ar[rr]^-{\rho}
	&& H \wedge_{H\bZ} \THH(\bZ, \langle p\rangle)
}
\end{split}\end{equation}
of commutative $H$-algebras, and an associated $\Tor$ spectral sequence,
as asserted.  Here
\[\begin{split}
E^2_{**} &= \Tor^{P(p) \otimes E(dp)}_{**}
	(E(\lambda_1) \otimes P(\mu_1), P(p) \otimes E(d\log p)) \\
& \cong \Tor^{E(dp)}_{**}(E(\lambda_1) \otimes P(\mu_1), E(d\log p)) \\
& \cong E(\lambda_1) \otimes P(\mu_1) \otimes E(d\log p) \otimes \Gamma([dp])
\end{split}\]
by change-of-rings and the observation that $E(dp)$ acts trivially
on $E(\lambda_1)\otimes P(\mu_1)$ and $E(d\log p)$.

To determine the differentials and multiplicative extensions in this
spectral sequence, we use naturality of the K{\"u}nneth spectral sequence
with respect to the map of homotopy cocartesian squares from~\eqref{xy:zn}
to~\eqref{xy:zmodp}, induced by the augmentation
\[
\bS[B^{\rep}\langle p\rangle] \longto \bS
\]
viewed as a map of $\bS[B^{\cy}\langle p\rangle]$-algebras.
The induced homomorphism of $E^2$-terms
\[
E(\lambda_1) \otimes P(\mu_1) \otimes E(d\log p) \otimes \Gamma([dp])
\longto
E(\lambda_1) \otimes P(\mu_1) \otimes E([p]) \otimes \Gamma([dp])
\]
takes $d\log p$ to $0$ and preserves the other generators.  The class
$d\log p$ is an infinite cycle, for filtration reasons, so the
differentials $d^p(\gamma_k[dp]) \doteq \lambda_1 \cdot \gamma_{k-p}[dp]$
for $k\ge p$ on the right hand side lift to the left hand side.
This leaves the $E^{p+1}$-term
\[
E^{p+1}_{**} = P(\mu_1) \otimes E(d\log p) \otimes P_p([dp]) \,,
\]
which must equal the $E^\infty$-term for filtration reasons.
The multiplicative extension $[dp]^p \doteq \mu_1$ on the right hand
side must also lift to the left hand side, completing the proof.
\qed \end{proof}

\begin{remark}
The completion map $\bZ \to \bZ_p$ induces an isomorphism from this
spectral sequence to the one discussed in \cite[Example~8.13]{Rognes_TLS},
converging to $\bar\pi_* \THH(\bZ_p, \langle p\rangle)$.  The argument
just given justifies the assertions about differentials and multiplicative
extensions that were made without proof in the cited example.
\end{remark}

\begin{corollary}
There is an isomorphism
\[
\bar\pi_* \THH(\bZ_p, \langle p\rangle) \cong \bar\pi_* \THH(\bZ_p|\bQ_p)
\]
of $\bar\pi_* \THH(\bZ_p)$-algebras, where the right hand side is as defined by 
Hesselholt-Madsen~\cite{Hesselholt-M_local_fields}. \qed
\end{corollary}

\section{Localization sequences for log \texorpdfstring{$\THH$}{THH} of periodic ring spectra}\label{sec:hty-cof-log-thh}
In this section we construct homotopy cofiber sequences relating
$\THH(A,M)$ and $\THH(A)$ for certain pre-log ring spectra $(A,M)$.
We begin by describing the commutative $\cJ$-space monoids $M$
that participate in these pre-log ring spectra.

\subsection{Repetitive commutative \texorpdfstring{$\cJ$}{J}-space
  monoids}\label{subsec:periodic-J-monoids}
The category $\cJ$ has a $\bZ$-grading defined by letting the degree
of an object $(\bld{n_1},\bld{n_2})$ be the difference $n=n_2-n_1\in
\bZ$. Morphisms in $\cJ$ preserve the degree. We let $\cJ_n$ be the
full subcategory of $\cJ$ whose objects have degree $n$. If $X$ is a
$\cJ$-space, then the restriction along the inclusion
$\iota_n\colon\cJ_n \to \cJ$ defines a $\cJ_n$-space $\iota_n^*X$,
which we call the \emph{$\cJ$-degree $n$ part} of~$X$. The
decomposition of $X_{h\cJ}$ into the components
$(\iota_n^*X)_{h\cJ_n}$ defines a $\bZ$-grading on $X_{h\cJ}$. We can
view this grading as a map $ X_{h\cJ} \to \bZ$. If $M$ is a
$\cJ$-space monoid, then $M_{h\cJ}$ is a simplicial monoid because
$(-)_{h\cJ}$ is monoidal, and $M_{h\cJ} \to \bZ$ is a monoid
homomorphism to the (discrete simplicial) monoid~$(\bZ,+)$.

\begin{definition}\label{def:J-space-components}
For each $\cJ$-space $X$ and each subset $S \subseteq \bZ$
let $X_S \subseteq X$ be the sub $\cJ$-space with
\[
X_S(\bn_1, \bn_2) =
\begin{cases}
X(\bn_1, \bn_2) & \text{if $n_2 - n_1 \in S$,} \\
\emptyset & \text{otherwise.}
\end{cases}
\]
In particular we have the part $X_{\{0\}}$ in $\cJ$-degree zero, the
part $X_{\neq 0}=X_{\bZ\setminus \{0\}}$ in non-zero $\cJ$-degrees,
the part $X_{>0} = X_{\bN}$ in positive $\cJ$-degrees, and the part
$X_{\ge0} = X_{\bN_0}$ in non-negative $\cJ$-degrees. If $X$ is the
cyclic or replete bar construction on a commutative $\cJ$-space monoid
$M$, we often denote $X_S$ by $B^{\cy}_S(M)$ or $B^{\rep}_S(M)$,
respectively.
\end{definition}
Let $T$ be the terminal $\cJ$-space, i.e., the constant functor $T\colon \cJ\to\cS$ with value a chosen one-point space. We can obtain the inclusions
$X_{\ge0} \to X$ and $ X_{\{0\}} \to X$ as the base changes along $X \to T$
of $T_{\ge0} \to T$ and $T_{\{0\}} \to T$, respectively. The last two maps
are maps of commutative $\cJ$-space monoids. This verifies:
\begin{lemma}
  If $M$ is a (commutative) $\cJ$-space monoid, then so are $M_{\{0\}}$ and
  $M_{\ge0}$, and the inclusions $M_{\{0\}} \to M_{\ge0} \to M$ are monoid
  maps.\qed
\end{lemma}
\begin{lemma}\label{lem:Mcof-M0cof}
If $M$ is a cofibrant commutative $\cJ$-space monoid such that $M_{<0} = \emptyset$, 
then $M_{\{0\}}$ is also cofibrant. 
\end{lemma}
\begin{proof}
  We may assume that $M$ is a cell complex obtained by attaching
  generating cofibrations of the positive $\cJ$-model structure.
  Since $\cJ$-spaces $X$ and $Y$ with $X_{<0} = \emptyset$ and $Y_{<0}
  = \emptyset$ satisfy $(X\boxtimes Y)_{\{0\}} \iso X_{\{0\}} \boxtimes Y_{\{0\}}$, it
  follows from the analysis of cell attachments
  in~\cite[Proposition 10.1]{Sagave-S_diagram} that $M_{\{0\}}$ is cofibrant. 
\qed \end{proof}
We are now ready to introduce the condition that will ensure the
existence of the homotopy cofiber sequence for logarithmic $\THH$ that we
are after.
\begin{definition}\label{def:repetitive}
  A commutative $\cJ$-space monoid $M$ is \emph{repetitive} if it is
  not concentrated in $\cJ$-degree zero and if the group completion
  map $M \to M^{\gp}$ induces a $\cJ$-equivalence $M \to (M^{\gp})_{\ge0}$. 
\end{definition}
It follows that a repetitive commutative $\cJ$-space monoid $M$
satisfies $M_{<0} = \emptyset$ and $M_{>0} \neq \emptyset$. 
If $M$ is a grouplike commutative $\cJ$-space monoid that is not
concentrated in $\cJ$-degree zero, then we say that $M$ is
\emph{$d$-periodic} if $d$ is the minimal positive element in the
image of the grading map $M_{h\cJ}\to\bZ$. In this case, $M_{\{n\}} = \emptyset$
if and only if~$d\nmid n$. 
\begin{definition}\label{def:repetitive-of-period}
A repetitive commutative $\cJ$-space monoid $M$ is
  \emph{repetitive of period~$d$} if its group completion $M^{\gp}$ is
  $d$-periodic.
\end{definition}

Thus, if $M$ is repetitive, then $M_{\{n\}}$ is non-empty if and only if $d \mid n$ and $n \geq 0$.
The next lemma shows that every grouplike commutative $\cJ$-space monoid
that is not concentrated in $\cJ$-degree zero gives rise to a repetitive
commutative $\cJ$-space monoid.
\begin{lemma}\label{lem:Mgeq0-repetetive}
  If $M$ is a grouplike commutative $\cJ$-space monoid of period $d>0$, then $M_{\geq 0}$ is repetitive of period $d$. 
\end{lemma}
\begin{proof}
  Since $M$ is grouplike, it follows that the grading map $M_{h\cJ}\to
  \bZ$ induces surjective monoid maps $(M_{\geq 0})_{h\cJ} \to d\bN_0$
  and $M_{h\cJ} \to d\bZ$.  Applying the simplicial bar construction,
  we get the following square of bisimplicial sets:
   \[\xymatrix@-1pc{
    B_{\bullet}((M_{\geq 0})_{h\cJ}) \ar[d] \ar[r] & B_{\bullet}(M_{h\cJ})\ar[d] \\
    B_{\bullet}(d\bN_0)\ar[r] & B_{\bullet}(d\bZ) \, . }\] By the
  Bousfield--Friedlander Theorem~\cite[Theorem
  B.4]{Bousfield-F_Gamma-bisimplicial}, the square induces a homotopy
  cartesian square after realization. Hence $B((M_{\geq 0})_{h\cJ}) \to
  B(M_{h\cJ})$ is a weak equivalence because $B(d\bN_0)\to B(d\bZ)$ is one,
  and consequently $M_{\geq 0} \to M$ is $\cJ$-equivalent to the group completion 
  $M_{\geq 0} \to (M_{\geq 0})^{\gp}$. 
\qed \end{proof}

Let $E$ be a commutative symmetric ring spectrum with $0\neq 1$ in
$\pi_0(E)$.  We say that $E$ is $d$-periodic if $\pi_*(E)$ has a unit
of non-zero degree and $d$ is the minimal positive degree of such a
unit.

\begin{corollary}\label{cor:istarGLoneJofE-repetitive}
  Let $E$ be a positive fibrant commutative symmetric ring spectrum
  that is $d$-periodic for some $d>0$, and let $j\colon e \to E$ be a
  model for the connective cover of $E$ with $e$ positive fibrant.
  Then the commutative $\cJ$-space monoid $j_*\!\GLoneJof(E)$ is
  repetitive of period~$d$.
\end{corollary}
\begin{proof}
  In this situation $\GLoneJof(E)$ is grouplike of period $d$ and the map $\Omega^{\cJ}(e \to E)_{\geq 0}$ is a
  $\cJ$-equivalence. Hence the map $j_*\!\GLoneJof(E) \to \GLoneJof(E)$ is
  $\cJ$-equivalent to the inclusion of $\GLoneJof(E)_{\geq 0}$.
\qed \end{proof}

\begin{example}\label{ex:Dx-repetitive}
  The commutative $\cJ$-space monoid $D(x)$ described in
  Example~\ref{ex:Dx} is repetitive of period $d$. 
\end{example}

\subsection{Homotopy cofiber sequences}
Let $M$ be a commutative $\cJ$-space monoid that is concentrated in non-negative $\cJ$-degrees, i.e., $M=M_{\geq 0}$. We then have a collapse map of commutative symmetric ring spectra 
\[
\pi\colon \bS^{\cJ}[M]\cong\bigvee_{n\geq 0} \bS^{\cJ}[M_{\{n\}}]\to \bS^{\cJ}[M_{\{0\}}]
\]
which is the identity on $\bS^{\cJ}[M_{\{0\}}]$ and takes each wedge summand $\bS^{\cJ}[M_{\{n\}}]$ for $n>0$ to the base point.

\begin{definition}\label{def:A-mod-M}
  Let $(A,M)$ be a pre-log ring spectrum with $M$ concentrated in
  non-negative $\cJ$-degrees. The \emph{quotient of $A$
    along the map $\pi\colon \bS^{\cJ}[M] \to \bS^{\cJ}[M_{\{0\}}]$} is the commutative
  $A$-algebra spectrum
  \[
  \modmod{A}{M} = A \wedge_{\bS^{\cJ}[M]}
  \bS^{\cJ}[M_{\{0\}}]
  \] 
  given by the pushout of $A \ot \bS^{\cJ}[M] \to
  \bS^{\cJ}[M_{\{0\}}]$ in $\cC\Spsym$. 
\end{definition}
We observe that if $(A,M)$ is a cofibrant pre-log ring spectrum, then
the adjoint structure map $\bar \alpha \colon \bS^{\cJ}[M] \to A$ is a
cofibration, and the pushout square defining $\modmod{A}{M}$ is
homotopy cocartesian since $\cC\Spsym$ is left proper. Furthermore,
$A/(M_{>0})$ is then cofibrant by Lemma~\ref{lem:Mcof-M0cof}.
  
The statement of the next theorem is as in
Theorem~\ref{thm:THH-cofiber-sequence-with-repetitive} from the
introduction, except that we make the cofibrancy condition on $(A,M)$
explicit.

\begin{theorem}\label{thm:repetitive-log-THH-homotopy-cofiber}
Let $(A,M)$ be a cofibrant pre-log ring spectrum with $M$ repetitive. Then there is a natural homotopy cofiber sequence
\[
\THH(A) \overset{\rho}{\longto}
\THH(A, M) \overset{\partial}{\longto}
\Sigma \THH(\modmod{A}{M})
\]
of $\THH(A)$-module spectra with circle action.
\end{theorem}
We proceed to explain how to derive this theorem from the general principle in Proposition~\ref{prop:general-logTHH-homotopy-cofiber}.  Evaluating the cyclic bar construction on $\pi\colon \bS^{\cJ}[M] \to \bS^{\cJ}[M_{\{0\}}]$ we get a map of commutative symmetric ring spectra  
that can be identified with the projection $\pi\colon \bS^{\cJ}[B^{\cy}(M)]\to \bS^{\cJ}[B^{\cy}(M_{\{0\}})]$, collapsing the positive part of $B^{\cy}(M)$ and mapping $\bS^{\cJ}[B^{\cy}_{\{0\}}(M)]$ isomorphically onto $\bS^{\cJ}[B^{\cy}(M_{\{0\}})]$. This applies in particular if $M$ is repetitive, and we shall view $\bS^{\cJ}[B^{\cy}(M_{\{0\}})]$ as an $\bS^{\cJ}[B^{\cy}(M)]$-module via $\pi$.

\begin{proposition}\label{prop:SJ-repetitive-homotopy-cofiber}
Let $M$ be a repetitive and cofibrant commutative $\cJ$-space monoid. Then there is a homotopy cofiber sequence 
\[
\bS^{\cJ}[B^{\cy}(M)]\overset{\rho}{\longto} \bS^{\cJ}[B^{\rep}(M)] \overset{\partial}{\longto} \Sigma\bS^{\cJ}[B^{\cy}(M_{\{0\}})]
\]
of $\bS^{\cJ}[B^{\cy}(M)]$-modules with circle action.
\end{proposition} 

The proof of this proposition requires a thorough analysis of the repletion map~$\rho$ and is postponed until Section~\ref{sec:proposition-proof}.

\begin{proof}[ 
of Theorem~\ref{thm:repetitive-log-THH-homotopy-cofiber}]
With the notation from Proposition~\ref{prop:general-logTHH-homotopy-cofiber}, we consider the cofibrant commutative $\cJ$-space monoid $P=M_{\{0\}}$ and the map of commutative symmetric ring spectra $\pi\colon \bS^{\cJ}[M] \to \bS^{\cJ}[M_{\{0\}}]$. 
The homotopy cofiber sequence in Proposition~\ref{prop:SJ-repetitive-homotopy-cofiber} now gives the necessary input for Proposition~\ref{prop:general-logTHH-homotopy-cofiber}, and the result follows. 
\qed \end{proof}

\subsection{Identification of the homotopy cofiber} 

Let $M$ be a commutative $\cJ$-space monoid with multiplication $\mu
\colon M \boxtimes M \to M$. We noticed in
Example~\ref{ex:free-J-space} that a $0$-simplex $x \in M(\bld{d_1},
\bld{d_2})_0$ corresponds to a map $\bar x \:
F^{\cJ}_{(\bld{d_1},\bld{d_2})}(*) \to M$ from the free $\cJ$-space on
a point in bidegree $(\bld{d_1},\bld{d_2})$. In the following, we
will refer to the composite map of $\cJ$-spaces 
\begin{equation}\label{eq:mult-with-x} F^{\cJ}_{(\bld{d_1},\bld{d_2})}(*) \boxtimes M \xrightarrow{\bar x
    \boxtimes \mathrm{id}} M \boxtimes M \xrightarrow{\mu} M
\end{equation}
as the \emph{multiplication with $x$}. 
\begin{lemma}\label{lem:grouplike-periodic}
  If $M$ is a grouplike commutative $\cJ$-space monoid, then the
  multiplication with any $0$-simplex  $x$ is a $\cJ$-equivalence. \end{lemma}
\begin{proof}
  The $0$-simplex in $(F^{\cJ}_{(\bld{d_1},\bld{d_2})}(*))_{h\cJ}$ specified
  by $(\bld{d_1},\bld{d_2})$, the monoidal structure
  map of $(-)_{h\cJ}$ and the multiplication with $x$ induce a
  sequence of maps
  \begin{equation}\label{eq:mult-by-x-on-MhJ}
    M_{h\cJ} \to \left(F^{\cJ}_{(\bld{d_1},\bld{d_2})}(*)\right)_{h\cJ} \times M_{h\cJ}
    \to \left(F^{\cJ}_{(\bld{d_1},\bld{d_2})}(*) \boxtimes M\right)_{h\cJ} \to M_{h\cJ}. 
  \end{equation}
  The space $(F^{\cJ}_{(\bld{d_1},\bld{d_2})}(*))_{h\cJ}$ is
  contractible because it is isomorphic to the classifying space of
  the comma category ($(\bld{d_1},\bld{d_2})\! \downarrow \!\cJ)$. So the
  first map in~\eqref{eq:mult-by-x-on-MhJ} is a weak equivalence. The
  second map is a weak equivalence by~\cite[Lemma
  2.11]{Sagave_log-on-k-theory}. Hence the multiplication with $x$ is
  a weak equivalence if and only if the composite map
  in~\eqref{eq:mult-by-x-on-MhJ} is one. The latter condition holds since
  $M$ is grouplike.
\qed \end{proof}

\begin{corollary}\label{cor:components-of-repetitive-repete}
  If $M$ is repetitive of period $d$, then the multiplication with $x$ induces a
  $\cJ$-equivalence
  \[ F^{\cJ}_{(\bld{d_1},\bld{d_2})}(*) \boxtimes M \to M_{> 0}
  \]
  for any $0$-simplex $x \in  M(\bld{d_1},\bld{d_2})$ of $\cJ$-degree $d$.
\end{corollary}
\begin{proof}
This is an immediate consequence of Lemma~\ref{lem:grouplike-periodic}.
\qed \end{proof}

We now explain how to identify $\modmod{A}{M}$ in examples. A first step is:
\begin{lemma}\label{lem:SJ-repetitive}
  Let $M$ be a cofibrant commutative $\cJ$-space monoid that is repetitive of
  period~$d$, and let $x \in M(\bld{d_1},\bld{d_2})$ be a $0$-simplex of
  $\cJ$-degree $d$. Then there is a homotopy cocartesian square of $\bS^{\cJ}[M]$-modules
 \[
\xymatrix@-1pc{
\bS^{\cJ}[F^{\cJ}_{(\bld{d_1},\bld{d_2})}(*) \boxtimes M] \ar[r] \ar[d] & \bS^{\cJ}[M] \ar[d]_{\pi} \\
{*} \ar[r] & \bS^{\cJ}[M_{\{0\}}],
}
\]
where the top horizontal map is induced by the multiplication with $x$.
\end{lemma}
\begin{proof}
  It suffices to show that the square is homotopy cocartesian as a
  diagram of symmetric spectra, and as such the top horizontal map
  factors as the composition
  \[
  \bS^{\cJ}[F^{\cJ}_{(\bld{d_1},\bld{d_2})}(*) \boxtimes M] \xrightarrow{\simeq} 
  \bS^{\cJ}[M_{> 0}]\to \bS^{\cJ}[M_{\{0\}}] \vee
  \bS^{\cJ}[M_{> 0}]\cong \bS^{\cJ}[M] \, .
  \]
  The result easily follows from the fact that the first map is a
  stable equivalence, by
  Corollaries~\ref{cor:components-of-repetitive-repete} and
  \ref{cor:bSJ-preserves-we}.
\qed \end{proof}

Let $(A,M)$ be a pre-log ring spectrum. The choice of a $0$-simplex $x \in
M(\bld{d_1},\bld{d_2})_0$ and the structure map $\alpha\colon M \to
\Omega^{\cJ}(A)$ determine a $0$-simplex in
\[\Omega^{\cJ}(A)(\bld{d_1},\bld{d_2}) =\Omega^{d_2}(A_{d_1})\cong
\mathrm{Map}_{\Spsym}(F_{d_1}S^{d_2},A),\] i.e., a map $F_{d_1}S^{d_2}
\to A$. Here $ F_{d_1}S^{d_2} \simeq \Sigma^{d_2-d_1}(\bS)$ is the
free symmetric spectrum on the space $S^{d_2}$ in degree
$d_1$. Extension of scalars provides an $A$-module map
$\widetilde{x}\colon F_{d_1}S^{d_2}\sm A \to A$ from the free
$A$-module spectrum on $S^{d_2}$ in degree $d_1$, which depends on $x$
and $\alpha$.

\begin{lemma}\label{lem:A-mod-M-as-hty-cofiber}
  Let $(A,M)$ be a cofibrant pre-log ring spectrum with $M$ repetitive of
  period~$d$. Then there is a homotopy cocartesian square of $A$-module spectra
\[\xymatrix@-1pc{ F_{d_1}S^{d_2} \sm A \ar[d] \ar[r]^-{\widetilde{x}} & A \ar[d] \\ {*} \ar[r]&  \modmod{A}{M},} \]
where the right hand vertical map is the canonical map  to the quotient
of $A$ along the collapse map $\pi\colon \bS^{\cJ}[M]\to \bS^{\cJ}[M_{\{0\}}]$.
\end{lemma}
\begin{proof}
Notice first that the diagram in the lemma is isomorphic to that obtained by applying the functor $A\wedge_{\bS^{\cJ}[M]}(-)$ to the diagram in Lemma~\ref{lem:SJ-repetitive}. Writing $C(\tilde x)$ for the mapping cone of $\tilde x$, it follows from Lemmas~\ref{lem:SJ-repetitive} and \ref{lem:ext-along-cof-pre-st-eq} that the canonical map
\[
C(\tilde x)\cong A\wedge _{\bS^{\cJ}[M]}C\left(\bS^{\cJ}[F^{\cJ}_{(\bld{d_1},\bld{d_2})}(*) \boxtimes M]\to \bS^{\cJ}[M]\right)\to A\wedge_{\bS^{\cJ}[M]}{\bS^{\cJ}[M_{\{0\}}]}
\]
is a stable equivalence. This is equivalent to the statement in the lemma.
\qed \end{proof}

\begin{lemma}\label{lem:A-mod-repetetive-Postn-section}
  Let $(A,M)$ be a cofibrant pre-log ring spectrum with $A$ connective
  and with $M$ repetitive of period $d$. Suppose that there is a $0$-simplex
  $x \in M(\bld{d_1},\bld{d_2})_0$ of $\cJ$-degree $d$ such that the
  homotopy class $[x] \in \pi_d(A)$ represented by the
  image of~$x$ in $\Omega^{\cJ}(A)$ has the property that
  multiplication with $[x]$ induces an isomorphism
  $\pi_i(A) \to \pi_{i+d}(A)$ for each $i \geq 0$.

  Then the map $\pi_i(A) \to \pi_i(\modmod{A}{M})$ is an isomorphism
  for $0\leq i < d$, and $\pi_i(\modmod{A}{M})$ is trivial for
  all other $i$.
\end{lemma}
\begin{proof}
  In the long exact sequence of homotopy groups associated with the
  homotopy cofiber sequence arising from
  Lemma~\ref{lem:A-mod-M-as-hty-cofiber}, the map $\widetilde{x}$
  induces multiplication with~$[x]$. The claim follows
  by inspection of the long exact sequence.
\qed \end{proof}

Now let $E$ be a $d$-periodic and positive fibrant commutative symmetric ring spectrum with connective cover $e\to E$, where $d>0$. In the following we shall consider the corresponding direct image log ring spectrum $(e,j_*\!\GLoneJof(E))$ from Definition~\ref{def:direct-image-log-str}. Notice that if $(e^{\mathrm{cof}},j_*\!\GLoneJof(E)^{\mathrm{cof}})$ denotes a cofibrant replacement of the latter, then the composition $e^{\mathrm{cof}}\to E$ is again a connective cover and $j_*\!\GLoneJof(E)^{\mathrm{cof}}\to j_*\!\GLoneJof(E)$ is a $\cJ$-equivalence. It will be convenient to simplify notation by writing $(e,j_*\!\GLoneJof(E))$ also for the cofibrant replacement when the meaning is clear from the context. 

\begin{corollary}
  Let $E$ be a $d$-periodic positive fibrant commutative symmetric
  ring spectrum with connective cover $j\colon e \to E$, and let $(e,j_*\!\GLoneJof(E))$ denote a cofibrant replacement of the corresponding direct image log ring spectrum. Then the commutative symmetric ring spectrum
  $\modmod{e}{j_*\!\GLoneJof(E)}$ associated with 
  $(e,j_*\!\GLoneJof(E))$ is stably equivalent to the
  $(d-1)$-th Postnikov section $\postnikovsec{e}{0}{d}$ of $e$.
\end{corollary}
\begin{proof}
  The commutative $\cJ$-space monoid $j_*\!\GLoneJof(E)$ is repetitive
  by Corollary~\ref{cor:istarGLoneJofE-repetitive}. By construction,
  the composite $ j_*\!\GLoneJof(E) \to \Omega^{\cJ}(e) \to
  \Omega^{\cJ}(E)$ maps any $0$-simplex $x \in
  (j_*\!\GLoneJof(E))(\bld{d_1},\bld{d_2})_0$ of $\cJ$-degree $d$ to a
  representative of a unit in $\pi_*(E)$. So the image of $x$ in
  $\Omega^{\cJ}(e)$ represents a homotopy class satisfying the
  assumptions of the previous lemma.
\qed \end{proof}
Combining this corollary with 
Theorem~\ref{thm:repetitive-log-THH-homotopy-cofiber} gives
the following result, which is similar to the statement in Theorem~\ref{thm:THH-cofiber-sequence-with-direct-image} from the introduction, except that we here make the fibrancy and cofibrancy conditions on $E$ and $(e,j_*\!\GLoneJof(E))$ explicit. We require $E$ to be positive fibrant to ensure that $\Omega^{\cJ}(E)$ and  $\GLoneJof(E)$ capture the desired homotopy type, so that $j_*\!\GLoneJof(E)$ is repetitive.   

\begin{theorem}\label{thm:hty-cofiber-on-log-THH-for-e-to-E}
  Let $E$ be a $d$-periodic positive fibrant commutative symmetric
  ring spectrum with connective cover $j\colon e \to E$, and let $(e,j_*\!\GLoneJof(E))$ denote a cofibrant replacement of the corresponding direct image log ring spectrum. Then there is a natural homotopy cofiber sequence
\[\THH(e) \overset{\rho}{\longto}
\THH(e,j_*\!\GLoneJof(E)) \overset{\partial}{\longto} \Sigma \THH(\postnikovsec{e}{0}{d})\] of $\THH(e)$-module spectra with circle
action. \qed
\end{theorem}
\begin{remark}\label{rem:more-repetetive}
  Currently we do not know if it is possible to find repetitive
  pre-log structures on ring spectra $A$ that do not arise as the
  connective covers of periodic ring spectra.  Such pre-log structures
  would lead to new examples of homotopy cofiber sequences for log
  $\THH$, where the term $\modmod{A}{M}$ might be more difficult to
  describe.
\end{remark}

\section{The proof of Proposition~\ref{prop:SJ-repetitive-homotopy-cofiber}}\label{sec:proposition-proof}
In this section we prove Proposition~\ref{prop:SJ-repetitive-homotopy-cofiber}, thereby completing the last step in setting up the homotopy cofiber sequence in 
Theorem~\ref{thm:repetitive-log-THH-homotopy-cofiber}. As a beginning, we observe that if $M$ is a commutative $\cJ$-space
monoid, then the $\bZ$-grading of $\cJ$ induces an augmentation
$B^{\cy}(M)_{h\cJ} \to B^{\cy}(\bZ)$ of the cyclic bar
construction. Here $ B^{\cy}(\bZ)$ denotes the cyclic bar construction
of the discrete monoid $\bZ$ as considered in
Section~\ref{subsec:cy-rep-discrete}. Using that $B^{\cy}(M)_{h\cJ}$
is isomorphic to the realization of the bisimplicial set $[q] \mapsto
B^{\cy}_{q}(M)_{h\cJ}$, the augmentation is defined as the realization
of the map
\[ B^{\cy}_q(M)_{h\cJ} \iso \coprod_{(x_0,\dots, x_q)
  \in B^{\cy}_{q}(\bZ)} (M_{\{x_0\}} \boxtimes \dots \boxtimes
M_{\{x_q\}})_{h\cJ} \to B_{q}^{\cy}(\bZ) \] that collapses $
(M_{\{x_0\}} \boxtimes \dots \boxtimes M_{\{x_q\}})_{h\cJ}$ to the
point $(x_0,\dots, x_q) \in B^{\cy}_{q}(\bZ)$. Notice that if we view $B^{\cy}(\bZ)$ as a constant commutative $\cJ$-space monoid, then there is an analogous augmentation $B^{\cy}(M)\to B^{\cy}(\bZ)$ before passing to the homotopy colimit. Using these augmentations, we can relate the repletion map for $B^{\cy}(M)$ to the repletion map for $B^{\cy}(\bZ)$ analyzed in 
Section~\ref{subsec:discrete-localization}.

\begin{proposition}\label{prop:Bcy-Brep-positive-equivalence}
  For a repetitive and cofibrant commutative
  $\cJ$-space monoid~$M$, the repletion map $\rho \: B^{\cy}(M) \to
  B^{\rep}(M)$ restricts to a $\cJ$-equivalence
\[
\rho_{>0}\colon B^{\cy}_{>0}(M) \xrightarrow{\sim} B^{\rep}_{>0}(M)
\]
in positive $\cJ$-degrees. 
\end{proposition}
\begin{proof}
  Because $M$ is repetitive, it follows from the definition of
  $B^{\rep}(M)$ that it is enough to show that the group completion $M
  \to M^{\gp}$ induces a $\cJ$-equivalence $B^{\cy}_{>0}(M) \to
  B^{\cy}_{>0}(M^{\gp})$. Assuming that $M$ is repetitive of period
  $d$, the augmentation introduced above induces a commutative square of
  bisimplicial sets
  \begin{equation}\label{eq:Bcy-P-Q-N-Z-hty-cart}\xymatrix@-1pc{
      B^{\cy}_{\bullet}(M)_{h\cJ} \ar[r] \ar[d] & B^{\cy}_{\bullet}(M^{\gp})_{h\cJ} \ar[d] \\
      B^{\cy}_{\bullet}(d\bN_0) \ar[r] & B^{\cy}_{\bullet}(d\bZ).  
    }\end{equation}
  Since $M$ is cofibrant, it follows from  Lemma~\ref{lem:properties-of-flat} and~\cite[Lemma
  2.11]{Sagave_log-on-k-theory} that there is a weak equivalence $B^{\cy}_{q}(M_{h\cJ})\to B^{\cy}_{q}(M)_{h\cJ}$, and similarly for $M^{\gp}$. This implies 
  that the square is a homotopy cartesian
  square of simplicial sets in every simplicial degree. To see that it is
  homotopy cartesian after realization, we use the
  Bousfield--Fried\-lander theorem~\cite[Theorem
  B.4]{Bousfield-F_Gamma-bisimplicial}: The bisimplicial sets
  $B^{\cy}_{\bullet}(M^{\gp})_{h\cJ}$ and $B^{\cy}_{\bullet}(d\bZ)$
  satisfy the $\pi_*$-Kan condition because $M^{\gp}$ and $d\bZ$
  are grouplike. Moreover, the map $ B^{\cy}_{\bullet}(M^{\gp})_{h\cJ}
  \to B^{\cy}_{\bullet}(d\bZ)$ is a Kan fibration on vertical path
  components because $B^{\cy}_{\bullet}(\pi_0(M^{\gp}_{h\cJ})) \to
  B^{\cy}_{\bullet}(d\bZ)$ is a degree-wise surjective homomorphism of
  simplicial abelian groups. This implies that~\cite[Theorem
  B.4]{Bousfield-F_Gamma-bisimplicial} applies.

  As we observed in Section~\ref{subsec:cy-rep-discrete},
  $B^{\cy}(d\bN_0) \to B^{\cy}(d\bZ)$ restricts to an equivalence on
  the positive components in the $\bZ$-grading. Hence $
  B^{\cy}_{>0}(M)_{h\cJ} \to B^{\cy}_{>0}(M^{\gp})_{h\cJ}$ is a weak
  equivalence because the realization
  of~\eqref{eq:Bcy-P-Q-N-Z-hty-cart} is homotopy cartesian. 
\qed \end{proof}

Next we observe that the group completion map $M \to M^{\gp}$ and the collapse map $\pi$ applied to the cyclic
and replete bar constructions give rise to the following diagram of commutative symmetric ring spectra with circle action
\[
\xymatrix@-1pc{
\bS^{\cJ}[B^{\cy}(M)] \ar[r]^-{\pi} \ar[d]_{\rho}& \bS^{\cJ}[B^{\cy}_{\{0\}}(M)] \ar@{=}[r] \ar[d]_{\rho_{\{0\}}}& \bS^{\cJ}[B^{\cy}(M_{\{0\}})]
\ar[r]^-{\simeq} &\bS^{\cJ}[B^{\cy}(M^{\gp}_{\{0\}})]\ar[d]_{\sigma}\\
\bS^{\cJ}[B^{\rep}(M)] \ar[r]^-{\pi} & \bS^{\cJ}[B^{\rep}_{\{0\}}(M)] \ar[rr]^-{\simeq} &&\bS^{\cJ}[B^{\cy}_{\{0\}}(M^{\gp})],
}
\]
where $\sigma$ is induced by the obvious inclusion $B^{\cy}(M^{\gp}_{\{0\}})\to B^{\cy}_{\{0\}}(M^{\gp})$. Here we are implicitly using Corollary~\ref{cor:bSJ-preserves-we} to ensure that the horizontal maps on the right hand side are stable equivalences as indicated. We shall view this as a commutative diagram of $\bS^{\cJ}[B^{\cy}(M)]$-modules with circle action.  

\begin{proposition}
The left hand square above is homotopy cocartesian in the category of $\bS^{\cJ}[B^{\cy}(M)]$-modules with circle action, hence induces a stable equivalence of the vertical mapping cones $\Cone(\rho)\xrightarrow{\simeq}\Cone(\sigma)$ as $\bS^{\cJ}[B^{\cy}(M)]$-modules with circle action. 
\end{proposition}
\begin{proof}
It suffices to show that the left hand square is homotopy cocartesian as a diagram of symmetric spectra. Clearly the map $\rho$ decomposes as the wedge sum of the restrictions $\rho_{\{0\}}$ and $\rho_{>0}$. The result therefore follows from Proposition~\ref{prop:Bcy-Brep-positive-equivalence} and Corollary~\ref{cor:bSJ-preserves-we} which together show that the latter map is a stable equivalence.
\qed \end{proof}
It remains to analyze the mapping cone $\Cone(\sigma)$. For this we
momentarily return to the cyclic bar construction $B^{\cy}(d\bZ)$ and
write $B^{\cy}_{\{0\}}(d\bZ)$ for the cyclic subobject with
$q$-simplices all tuples $(x_0, x_1, \dots, x_q)$ of elements in
$d\bZ$, subject to the condition that $x_0 + x_1 + \dots + x_q = 0$.
Here the notation is supposed to emphasize the fact that the
augmentation of $B^{\cy}(M^{\gp})$ discussed before
Proposition~\ref{prop:Bcy-Brep-positive-equivalence} restricts to a
map $B^{\cy}_{\{0\}}(M^{\gp})\to B^{\cy}_{\{0\}}(d\bZ)$ if $M$ is repetitive
of period $d$.

Writing $\cB(d\bZ)$ for the
groupoid with one object $*$ and automorphism group $d\bZ$, we may identify 
$B^{\cy}(d\bZ)$ with the cyclic nerve of $\cB(d\bZ)$.  Let
$\bZ/2 = \{0,1\}$ be the cyclic group of order two, and let
$\cE(\bZ/2)$ be the groupoid with objects $0$ and $1$, and a unique
morphism from each object to each object. We use the notation $E\bZ/2$ for the nerve of $\cE(\bZ/2)$. This is isomorphic to the cyclic nerve of $\cE(\bZ/2)$, since the zeroth component of a simplex in the cyclic nerve of a category is redundant if every object is initial.

Let $\phi \: \cE(\bZ/2) \to \cB(d\bZ)$ be the functor that takes the
morphism $0 \to 1$ to $d \: * \to *$, the morphism $1 \to 0$ to $-d \:
* \to *$, and the identity morphisms $0 \to 0$ and $1 \to 1$ to the
identity morphism $0 \: * \to *$.  With the notation introduced above,
the cyclic nerve of $\phi$ is a map $E\bZ/2 \to B^{\cy}(d\bZ)$.  Its
image $C$ is the simplicial subset of $B^{\cy}(d\bZ)$ whose
$q$-simplices are tuples $(x_0, x_1, \dots, x_q)$ of elements in
$\{d,0,-d\}$ satisfying the condition that the signs of the nonzero
terms $x_i$ alternate cyclically.  This is a cyclic subobject of
$B^{\cy}_{\{0\}}(d\bZ)$ that fits in the diagram of cyclic sets
\begin{equation}\label{eq:EZ2-Bcy0dZ-decomposition}
\xymatrix@-1pc{
&& \bZ/2 \ar@{->>}[rr] \ar@{ >->}[d] && \{(0)\} \ar@{ >->}[d] \\
\{0\} \ar@{ >->}[rr] 
&& E\bZ/2 \ar@{->>}[rr]
&& C \ar@{ >->}[rr] 
&& B^{\cy}_{\{0\}}(d\bZ),
}
\end{equation}
where $\bZ/2 = \{0, 1\}$ is viewed as the discrete cyclic set on the
set of $0$-simplices in $E\bZ/2$ and $\{(0)\}$ is the set of
$0$-simplices of $C$. It is immediate from the definition that the maps are
degreewise injective and surjective as indicated.

\begin{lemma}\label{lem:EZ2-Bcy0dZ-decomposition}
The square in diagram~\eqref{eq:EZ2-Bcy0dZ-decomposition} is a pushout, 
and the two remaining maps, $\{0\} \to E\bZ/2$ and $C  \to B^{\cy}_{\{0\}}(d\bZ)$,
are weak equivalences. 
\end{lemma}
\begin{proof}
  It is easy to see that the square is a pushout in every simplicial
  degree and hence a pushout of simplicial sets. Since $\cE(\bZ/2)$ has an initial object its nerve $E\bZ/2$ is contractible and hence $C$ has the homotopy type
  of $\Sigma\{0,1\} \iso S^1$. The map $C \to  B^{\cy}_{\{0\}}(d\bZ) \simeq S^1$
  is a weak equivalence since it maps the generator $[(-d,d)]$ of $\pi_1(C)$ to 
  a generator of $\pi_1(B^{\cy}_{\{0\}}(d\bZ))$. 
\qed \end{proof}
Consider in general a $d$-periodic grouplike and cofibrant commutative $\cJ$-space monoid $N$ with $d>0$. The idea for the next step is to pull $B^{\cy}_{\{0\}}(N)$ back over the
diagram~\eqref{eq:EZ2-Bcy0dZ-decomposition} using the structure map to
$B^{\cy}_{\{0\}}(d\bZ)$. It will be convenient to use the following ad hoc notation.  
For a map of cyclic sets $\theta \colon K \to
B^{\cy}(d\bZ)$, we write $\theta(k) = (\theta(k)_0,\dots, \theta(k)_q)$ for the image of a $q$-simplex $k \in K_q$ and
let $K \odot N$ denote the cyclic $\cJ$-space with
\[ 
(K \odot N)_q = \coprod_{k\in K_q}N_{\{\theta(k)_0\}}\boxtimes \dots \boxtimes N_{\{\theta(k)_q\}}
\]
and the obvious structure maps.

\begin{lemma}\label{lem:J-equiv-on-odot-N}
If $K \to L \to B^{\cy}(d\bZ)$ are maps of cyclic sets such that $K\to L$ is
a weak equivalence on underlying simplicial sets, then the induced map
of $\cJ$-spaces $K\odot N \to L \odot N$ is a $\cJ$-equivalence.  
\end{lemma}
\begin{proof}
We consider the induced commutative diagram of bisimplicial sets
\[\xymatrix@-1pc{
(K\odot N)_{h\cJ} \ar[r] \ar[d] & (L\odot N)_{h\cJ} \ar[r] \ar[d] & B^{\cy}(N)_{h\cJ} \ar[d] \\ K \ar[r] & L \ar[r] & B^{\cy}(d\bZ).
}
\]
The right hand square and the outer square are homotopy cartesian
in every simplicial degree $q$ because they are actual pullbacks and
$K_q \to L_q \to B^{\cy}_q(d\bZ)$ are maps of discrete simplicial
sets. As in the proof of
Proposition~\ref{prop:Bcy-Brep-positive-equivalence}, it follows that
these two squares are homotopy cartesian after realization. Hence the
left hand square is also homotopy cartesian after realization, and $K\odot N
\to L\odot N$ is a $\cJ$-equivalence since $K \to L$ is a weak
equivalence.
\qed \end{proof}

With notation from diagram~\eqref{eq:EZ2-Bcy0dZ-decomposition} we have the isomorphisms
\[
\{0\}\odot N \iso B^{\cy}(N_{\{0\}}), \quad \bZ/2\odot N \iso \bZ/2\times B^{\cy}(N_{\{0\}}),\quad B^{\cy}_{\{0\}}(d\bZ)\odot N \iso B^{\cy}_{\{0\}}(N)
\]
of cyclic $\cJ$-spaces. Pulling $B^{\cy}_{\{0\}}(N)$ back over diagram~\eqref{eq:EZ2-Bcy0dZ-decomposition} we therefore get the following commutative diagram of cyclic $\cJ$-spaces 
\[
  \xymatrix@-1pc{
    && \bZ/2 \times  B^{\cy}(N_{\{0\}}) \ar[rr] \ar[d]  
    &&  B^{\cy}(N_{\{0\}}) \ar[d]  
    \\
    B^{\cy}(N_{\{0\}}) \ar[rr]^{\simeq} 
    && E\bZ/2 \odot N \ar[rr]
    && C \odot N \ar[rr]^{\simeq} 
    && B^{\cy}_{\{0\}}(N),
    }
\]
where the square is a pushout by construction and the remaining horizontal maps are $\cJ$-equivalences by Lemmas~\ref{lem:EZ2-Bcy0dZ-decomposition} and \ref{lem:J-equiv-on-odot-N}. Applying the functor $\bS^{\cJ}$ and passing to the geometric realizations this in turn gives a commutative diagram
\begin{equation}\label{eq:SJ-EZ2-Bcy0N-decomposition}
  \xymatrix@-1pc{
    & \bS^{\cJ}[\bZ/2 \times  B^{\cy}(N_{\{0\}})] \ar[rr] \ar[d]  
    &&  \bS^{\cJ}[B^{\cy}(N_{\{0\}})] \ar[d]  
    \\
    \bS^{\cJ}[B^{\cy}(N_{\{0\}})] \ar[r]^-{\simeq} 
    & \bS^{\cJ}[E\bZ/2 \odot N] \ar[rr]
    && \bS^{\cJ}[C \odot N] \ar[r]^-{\simeq} 
    & \bS^{\cJ}[B^{\cy}_{\{0\}}(N)]
    }
\end{equation}
of $ \bS^{\cJ}[B^{\cy}(N_{\{0\}})]$-modules with circle action. The square is again a pushout square and the remaining vertical maps are stable equivalences by Corollary~\ref{cor:bSJ-preserves-we}. We claim that the square is in fact homotopy cocartesian and for this it suffices to observe that the vertical map on the left is a levelwise cofibration. Indeed, before geometric realization it is clearly a levelwise cofibration (in fact the inclusion of a wedge summand) in each simplicial degree and the geometric realization is therefore also a levelwise cofibration. Using this we can analyze the mapping cone $C(\sigma)$ of the composite map $\sigma\colon \bS^{\cJ}[B^{\cy}(N_{\{0\}})]\to \bS^{\cJ}[B^{\cy}_{\{0\}}(N)]$ in the diagram.

\begin{proposition}\label{prop:cone-SJBcyN0-Bcy0N}
Let $N$ be a grouplike and cofibrant commutative \mbox{$\cJ$-space} monoid with period $d>0$. Then the mapping cone $C(\sigma)$ is related to $\Sigma\bS^{\cJ}[B^{\cy}(N_{\{0\}})]$ by a chain of stable equivalences of $\bS^{\cJ}[B^{\cy}(N_{\{0\}})]$-modules with circle action.
\end{proposition}
\begin{proof}
Diagram~\eqref{eq:SJ-EZ2-Bcy0N-decomposition} gives rise to a commutative
diagram of mapping cones 
\[
\xymatrix@-1.3pc{
& C\big(\bS^{\cJ}[\bZ/2 \times B^{\cy}(N_{\{0\}})] \to \bS^{\cJ}[B^{\cy}(N_{\{0\}})]\big) \ar[d]^-{\simeq} \\
C\big(\bS^{\cJ}[B^{\cy}(N_{\{0\}})] \to \bS^{\cJ}[C \odot N]\big) \ar[r]^-{\simeq} \ar[d]^-{\simeq}
	& C\big(\bS^{\cJ}[E\bZ/2 \odot N] \to \bS^{\cJ}[C \odot N]\big) \ar[d]^-{\simeq} \\
C\big(\bS^{\cJ}[B^{\cy}(N_{\{0\}})] \to \bS^{\cJ}[B^{\cy}_{\{0\}}(N)]\big) \ar[r]^-{\simeq}
	& C\big(\bS^{\cJ}[E\bZ/2 \odot N] \to \bS^{\cJ}[B^{\cy}_{\{0\}}(N)]\big)
}
\]
where each arrow between mapping cones is a stable equivalence of $\bS^{\cJ}[B^{\cy}(N_{\{0\}})]$-modules with circle action.  The result follows by identifying $C(\sigma)$ with the lower left hand term and $\Sigma\bS^{\cJ}[B^{\cy}(N_{\{0\}})]$ with the upper right hand term.
\qed \end{proof}

\begin{proof}[ 
of Proposition~\ref{prop:SJ-repetitive-homotopy-cofiber}]
Let $N=M^{\gp}$. Pulling the stable equivalences in 
Proposition~\ref{prop:cone-SJBcyN0-Bcy0N} back to stable equivalences of $\bS^{\cJ}[B^{\cy}(M)]$-modules, we get the desired chain of stable equivalences 
\[
C(\rho)\xrightarrow{\simeq}C(\sigma)\simeq \Sigma\bS^{\cJ}[B^{\cy}(N_{\{0\}})]\xleftarrow{\simeq}
\Sigma\bS^{\cJ}[B^{\cy}(M_{\{0\}})]
\]
of $\bS^{\cJ}[B^{\cy}(M)]$-modules with circle action.
\qed \end{proof}

\begin{appendix}
\section{Homotopy invariance of \texorpdfstring{$\bS^{\cJ}$}{SJ}}\label{sec:Sj-homotopy-invariance}
Being a left Quillen functor, $\bS^{\cJ}$ takes $\cJ$-equivalences
between cofibrant $\cJ$-spaces to stable equivalences. It is useful to
know that $\bS^{\cJ}$ is homotopically well-behaved on a larger class
of $\cJ$-spaces that includes cofibrant $\cJ$-spaces and the
underlying $\cJ$-spaces of cofibrant commutative $\cJ$-space
monoids. For this purpose it is not sufficient to work with flat
$\cJ$-spaces, because $\bS^{\cJ}$ is not left Quillen with respect to
the flat model structure~\cite[Remark 4.29]{Sagave-S_diagram}.

\begin{definition}\label{def:SJ-good}
  A $\cJ$-space $X$ is \emph{$\mathbb S^{\cJ}\!$-good} if there exists
  a cofibrant $\cJ$-space $X'$ and a $\cJ$-equivalence $X'\to X$ such
  that $\mathbb S^{\cJ}[X']\to \mathbb S^{\cJ}[X]$ is a stable
  equivalence.
\end{definition}

It is clear from the definition that cofibrant $\cJ$-spaces are
$\mathbb S^{\cJ}\!$-good. The terminal $\cJ$-space $T$ is an example
of a $\cJ$-space that is not $\mathbb S^{\cJ}\!$-good. Using that
$\mathbb S^{\cJ}$ is a left Quillen functor we see that if $X$ is
$\mathbb S^{\cJ}\!$-good and $Y\to X$ is any $\cJ$-equivalence with
$Y$ cofibrant, then the induced map $\mathbb S^{\cJ}[Y]\to \mathbb
S^{\cJ}[X]$ is a stable equivalence. This in turn has the following
consequence.
\begin{proposition}\label{prop:SJ-good-equivalence}
  The functor $\mathbb S^{\cJ}$ takes $\cJ$-equivalences between
  $\mathbb S^{\cJ}\!$-good $\cJ$-spaces to stable equivalences. \qed
\end{proposition}

The automorphism group of an object $(\bld n_1,\bld n_2)$
in $\cJ$ may evidently be identified with $\Sigma_{n_1}\times \Sigma_{n_2}$.
\begin{definition}A $\cJ$-space $X$ is \emph{$\Sigma$-free in the
    second variable} if $\Sigma_{n_2}$ acts freely on $X(\bld n_1,\bld
  n_2)$ for each object $(\bld n_1,\bld n_2)$ in $\cJ$.
\end{definition}

\begin{lemma}
  If a $\cJ$-space is $\Sigma$-free in the second variable, then it is
  $\bS^{\cJ}\!$-good.
\end{lemma}
\begin{proof}
  Let $X$ be $\Sigma$-free in the second variable and let $X' \to X$
  be a cofibrant replacement, which we may assume is a level
  equivalence. Then $X'$ is also $\Sigma$-free in the second
  variable. The freeness assumptions imply that the quotients by the
  $\Sigma_k$-actions arising in the explicit description of
  $\bS^{\cJ}$ given in~\eqref{eq:OmegaJ-and-SJ-explicitly} preserve
  weak equivalences.  Hence $\bS^{\cJ}[X'] \to \bS^{\cJ}[X]$ is a
  level equivalence of symmetric spectra and therefore a stable
  equivalence.
\qed \end{proof}

The following condition for $\bS^{\cJ}$-goodness can often be checked
in practice.
\begin{corollary}\label{cor:Sigma-free-if-target-is}
  Let $X\to Y$ be a map of $\cJ$-spaces such that $Y$ is $\Sigma$-free
  in the second variable. Then $X$ is $\bS^{\cJ}$-good.
\end{corollary}
\begin{proof}
  If $Y$ is $\Sigma$-free in the second variable then it is automatic
  that also $X$ is $\Sigma$-free in the second variable. 
\qed \end{proof}

\begin{lemma}\label{lem:cof-CSJ-SJ-good}
  Let $M$ be a cofibrant commutative $\cJ$-space monoid. Then $M$ is
  $\Sigma$-free in the second variable and hence $\bS^{\cJ}$-good.
\end{lemma}
\begin{proof}
  Let us say that a $\cJ$-space $X$ is \emph{strongly free in the
    second variable} if for every subgroup $G \subseteq
  \Sigma_{m_1}\times \Sigma_{m_2}$ such that the composite $G \to
  \Sigma_{m_1}\times \Sigma_{m_2} \to \Sigma_{m_1}$ with the
  projection is injective, the group $\Sigma_{n_2}$ acts freely on
  \begin{equation}\label{eq:strongly-Sigma-cof} 
      (X \boxtimes (F^{\cJ}_{(\bld{m_1},\bld{m_2})}(*)/G))(\bld{n_1},\bld{n_2})\, .
  \end{equation}
  We will prove the lemma by showing the stronger statement that the
  underlying $\cJ$-space of $M$ is strongly free in the second
  variable.

  We first show that $U^{\cJ} = \cJ((\bld{0},\bld{0}),-)$ is strongly
  free in the second variable. Let $G \subseteq
  \Sigma_{m_1}\times\Sigma_{m_2}$ be a subgroup with $G \to
  \Sigma_{m_1}\times\Sigma_{m_2} \to \Sigma_{m_1}$ injective. If a morphism
  $(\alpha_1,\alpha_2,\rho)\colon
  (\bld{m_1},\bld{m_2})\to(\bld{n_1},\bld{n_2})$ represents an element
  \[[(\alpha_1,\alpha_2,\rho)] \in
  \cJ((\bld{m_1},\bld{m_2}),(\bld{n_1},\bld{n_2}))/G \iso (U^{\cJ} \boxtimes (F^{\cJ}_{(\bld{m_1},\bld{m_2})}(*)/G))(\bld{n_1},\bld{n_2}),\] then $\sigma
  [(\alpha_1,\alpha_2,\rho)] = [(\alpha_1,\alpha_2,\rho)]$ for a
  $\sigma \in \Sigma_{n_2}$ implies that there is a
  $(\gamma_1,\gamma_2)\in G$ with
  \[(\mathrm{id}_{\bld{n_1}},\sigma,\mathrm{id}_{\emptyset})(\alpha_1,\alpha_2,\rho)=(\alpha_1,\alpha_2,\rho)(\gamma_1,\gamma_2,\mathrm{id}_{\emptyset}).\]
  By definition of the composition in~$\cJ$ (see~\cite[Definition
  4.2]{Sagave-S_diagram}), this implies $\alpha_1 = \alpha_1
  \gamma_1$.  Since $\alpha_1$ is injective, $\gamma_1 =
  \mathrm{id}_{\bld{m_1}}$. Hence $\gamma_2 = \mathrm{id}_{\bld{m_2}}$
  because $G \to \Sigma_{m_1}$ is injective. So we have
  $(\mathrm{id}_{\bld{n_1}},\sigma,\mathrm{id}_{\emptyset})(\alpha_1,\alpha_2,\rho)=(\alpha_1,\alpha_2,\rho)$. This
  implies $\sigma(i) = i$ for $i\in \alpha_2(\bld{m_2})$. In the third
  variable, we have
  $\left(\sigma|_{\bld{n_2}\setminus\alpha_2}\right)\rho = \rho$ and
  hence $\sigma(i) = i$ for every $i\in \bld{n_2}\setminus
  \alpha_2$. Hence the $\Sigma_{n_2}$-action on $(U^{\cJ} \boxtimes
  (F^{\cJ}_{(\bld{m_1},\bld{m_2})}(*)/G))(\bld{n_1},\bld{n_2}) $ is
  free.

  Now we assume that $f \colon X \to Y$ is a generating cofibration
  for the positive $\cJ$-model structure on $\cS^{\cJ}$ and that the
  square
  \[\xymatrix@-1pc{
    \bC(X) \ar[r] \ar[d] & \bC(Y) \ar[d] \\
    A \ar[r] & B }\] is a pushout in~$\cC\cS^{\cJ}$. We want to show
  that the underlying $\cJ$-space of $B$ is strongly free in the
  second variable if that of $A$ is. For this we use~\cite[Proposition
  10.1]{Sagave-S_diagram}. It provides a filtration $A=F_0(B) \to
  F_1(B) \to \dots $ of $A \to B$ with $\colim_i F_i(B) = B$ such that
  there are pushout squares of
  $\cJ$-spaces \begin{equation}\label{eq:filtration-pushout}\xymatrix@-1pc{
      A\boxtimes (Q^i_{i-1}(f)/\Sigma_i) \ar[r] \ar[d] &
      A\boxtimes (Y^{\boxtimes i}/\Sigma_i) \ar[d]\\
      F_{i-1}(B) \ar[r]&F_{i}(B), }\end{equation} where $Q^i_{i-1}(f)
  \to Y^{\boxtimes i}$ is the iterated pushout product map associated
  with~$f$. Here we use that since we only consider the commutative
  case, the functors $U_i^{\bC}$ appearing in~\cite[Proposition
  10.1]{Sagave-S_diagram} are the forgetful functors (see
  \cite[Example 10.2]{Sagave-S_diagram}). Since the top horizontal map
  in~\eqref{eq:filtration-pushout} is a monomorphism of $\cJ$-spaces
  by~\cite[Proposition 7.1(vii)]{Sagave-S_diagram}, it is enough to
  show that $A\boxtimes (Y^{\boxtimes i}/\Sigma_i)$ is strongly free
  in the second variable. Using that $Y$ is of the form
  $F^{\cJ}_{(\bld{k_1},\bld{k_2})}(K)$ with $k_1 \geq 1$, it is
  therefore enough to show that $A\boxtimes
  ((F^{\cJ}_{(\bld{k_1},\bld{k_2})}(*))^{\boxtimes i}/\Sigma_i) $ is
  strongly free in the second variable if $k_1 \geq 1$. This follows
  from the hypothesis on $A$ because
  \[(F^{\cJ}_{(\bld{k_1},\bld{k_2})}(*)^{\boxtimes i}/\Sigma_i)
  \boxtimes (F^{\cJ}_{(\bld{m_1},\bld{m_2})}(*)/G) \iso
  F^{\cJ}_{(\bld{k_1},\bld{k_2})^{\concat i}\concat
    (\bld{m_1},\bld{m_2})}(*)/(\Sigma_i \times G)\] and $\Sigma_i
  \times G \to \Sigma_{ik_1+m_1}$ is injective if $k_1\geq 1$ and $G
  \to \Sigma_{m_1}$ is injective.

  For a general cofibrant commutative $\cJ$-space monoid $M$, we may
  without loss of generality assume that $M$ is a cell complex
  constructed from the generating cofibrations. This means that there
  is a $\lambda$-sequence $\{M_{\alpha} \colon \alpha < \lambda\}$ in
  $\cC\cS^{\cJ}$ for some ordinal $\lambda$ such that $M_0 = U^{\cJ}$
  and $M_{\alpha} \to M_{\alpha+1}$ is the cobase change of a
  generating cofibration in~$\cC\cS^{\cJ}$. In this situation, the
  above arguments imply that $M$ is strongly free in the second variable. 
\qed \end{proof}
The following consequence of Lemma~\ref{lem:cof-CSJ-SJ-good} can also
easily be verified directly.
\begin{corollary}\label{cor:cof-SJ-SJ-good}
  Positive cofibrant $\cJ$-spaces are $\Sigma$-free in the second variable.
\end{corollary}
\begin{proof}
  If $X$ is a positive cofibrant $\cJ$-space, then $\bC(X)$ is a
  cofibrant commutative $\cJ$-space monoid. Since there is a canonical
  map of $\cJ$-spaces $X \to \bC(X)$, the result follows by
  Lemma~\ref{lem:cof-CSJ-SJ-good} and the proof of
  Corollary~\ref{cor:Sigma-free-if-target-is}.  \qed \end{proof}
Combining Proposition~\ref{prop:SJ-good-equivalence},
Corollary~\ref{cor:Sigma-free-if-target-is} and
Lemma~\ref{lem:cof-CSJ-SJ-good} provides the following result.
\begin{corollary}\label{cor:bSJ-preserves-we}
  If $M$ is cofibrant in $\cC\cS^{\cJ}$ and $X \to Y \to M$ is a
  sequence of maps of $\cJ$-spaces with $X \to Y$ a $\cJ$-equivalence,
  then $\bS^{\cJ}[X]\to\bS^{\cJ}[Y]$ is a stable equivalence.\qed
\end{corollary}

\begin{remark}
  When working with the topological version of $\cJ$-spaces, the
  $\Sigma$-freeness condition should be replaced by a suitable
  equivariant cofibrancy condition. The analogue of
  Lemma~\ref{lem:cof-CSJ-SJ-good} then continues to hold, but we do not
  have a direct topological analogue of
  Corollary~\ref{cor:Sigma-free-if-target-is}.
\end{remark}
\end{appendix}

\end{document}